\documentclass[11pt, reqno]{amsart}
\usepackage[utf8]{inputenc}
\usepackage[mathscr]{euscript}
\usepackage{a4wide}
\usepackage[english]{babel}
\usepackage{amsfonts,amsmath,amsthm,amssymb,latexsym}
\usepackage{mathtools}
\usepackage{bm}
\usepackage{bbm}
\usepackage{graphicx}
\usepackage{esint}
\usepackage{enumitem}
\usepackage{caption}
\usepackage{xcolor}
\usepackage{multicol}
\usepackage{bm}
\usepackage[]{appendix}  
\usepackage{epsfig}
\usepackage{epstopdf}

\usepackage{cases}
\usepackage{pgfplots}
\usepackage{dsfont}
\usepackage{comment}
\usepackage{multicol}  
\usepackage{textcomp}
\usepackage{pgfplots}
\usepackage[numbers,sort&compress]{natbib}
\usepackage[colorlinks=true]{hyperref}
\hypersetup{urlcolor=blue, citecolor=red}
\DeclarePairedDelimiter{\abss}{\bigg\lvert}{\bigg\rvert}

\makeatletter
\newtheorem*{rep@theorem}{\rep@title}
\newcommand{\newreptheorem}[2]{%
	\newenvironment{rep#1}[1]{%
		\def\rep@title{#2 \ref{##1}}%
		\begin{rep@theorem}}%
		{\end{rep@theorem}}}
\makeatother

\theoremstyle{definition}
\newtheorem{rmk}{Remark}[section]
\newtheorem{definition}{Definition}[section]
\theoremstyle{plain}
\newtheorem{theorem}{Theorem}[section] 
\newreptheorem{theorem}{Theorem}

\newtheorem{prop}{Proposition}[section] 

\newcommand{\ds}{\displaystyle}

\newcommand{\partialt}[1]{\dfrac{\partial#1}{\partial t}}
\newcommand{\partialx}[1]{\dfrac{\partial#1}{\partial x}}
\newcommand{\partialv}[1]{\dfrac{\partial#1}{\partial v}}

\DeclareMathAlphabet{\mathup}{OT1}{\familydefault}{m}{n}
\newcommand{\dx}[1]{\mathop{}\!\mathup{d} #1}
\newcommand{\ddt}{{\frac{\dx{}}{\dx t}}}
\newcommand{\dxv}{\dx x\dx v}
\newcommand{\dxvt}{\dx x\dx v\dx t}
\newcommand{\dtxv}{\dx t\dx x\dx v}

\newcommand{\upsp}{\brk{\prt{\Upsilon_p}\i\n}}

\newcommand{\upsrhoh}{\prt{\Upsilon_f}_h}
\newcommand{\upsetah}{\prt{\Upsilon_g}_h}

\newcommand{\into}{\int_\Omega}

\DeclarePairedDelimiter{\prt}{(}{)}
\DeclarePairedDelimiter{\brk}{[}{]}

\DeclarePairedDelimiter{\abs}{|}{|}
\DeclarePairedDelimiter{\norm}{\|}{\|}
\DeclarePairedDelimiter{\set}{\{}{\}}

\newcommand{\n}{^{n}}
\newcommand{\nhalf}{1/2}
\renewcommand{\i}{_{i}}
\newcommand{\ip}{_{i+1}}
\newcommand{\im}{_{i-1}}
\newcommand{\ih}{_{i+\nhalf}}
\newcommand{\imh}{_{i-\nhalf}}\newcommand{\np}{^{n+1}}
\renewcommand{\j}{_{j}}
\newcommand{\jp}{_{j+1}}
\newcommand{\jm}{_{j-1}}
\newcommand{\jh}{_{j+\nhalf}}
\newcommand{\jmh}{_{j-\nhalf}}

\renewcommand{\ij}{_{i,j}}
\newcommand{\ijp}{_{i,j+1}}
\newcommand{\ijm}{_{i,j-1}}

\newcommand{\ipj}{_{i+1,j}}
\newcommand{\imj}{_{i-1,j}}

\newcommand{\lone}{^{\lambda_1}}
\newcommand{\ltwo}{^{\lambda_2}}

\newcommand{\pii}{K_{11}}
\newcommand{\pij}{K_{12}}
\newcommand{\pjj}{K_{22}}
\newcommand{\pji}{K_{21}}

\newcommand{\intt}{\int_{t\n}^{t\np}}
\newcommand{\intx}{\int_{x\imh}^{x\ih}}
\newcommand{\intv}{\int_{v\jmh}^{v\jh}}

\newcommand{\Cij}{{\abs{C_{i,j}}}}
\newcommand{\Cijn}{{\abs{C_{i,j}^n}}}

\newcommand{\Fx}{{\prescript{x}{}{\!\bar F}}}
\newcommand{\Fv}{{\prescript{v}{}{\!\bar F}}}
\newcommand{\Gx}{{\prescript{x}{}{\!\bar G}}}
\newcommand{\Gv}{{\prescript{v}{}{\!\bar G}}}
\newcommand{\Px}{{\prescript{x}{}{\!\bar P}}}
\newcommand{\Pv}{{\prescript{v}{}{\!\bar P}}}

\newcommand{\R}{\mathbb{R}}
\newcommand{\N}{\mathbb{N}}
\newcommand{\Z}{\mathbb{Z}}
\newcommand{\I}{\mathcal{I}}
\newcommand{\J}{\mathcal{J}}
\newcommand{\C}{\mathcal{C_W}}

\makeatletter
\@namedef{subjclassname@2020}{\textup{2020} Mathematics Subject Classification}
\makeatother

\begin{document}
\title{A Convergent Finite Volume Method for a Kinetic Model for Interacting Species}
\author{Julia I.M.~Hauser, Valeria Iorio, Markus Schmidtchen}

\address[Julia I.M.~Hauser]{\newline Institute of Scientific Computing, Technische Universit\"at Dresden, Zellescher Weg 12-14, 01069
Dresden, Germany}
\email{julia.hauser@tu-dresden.de}
\address[Valeria Iorio]{\newline DISIM - Department of Information Engineering, Computer Science and Mathematics, University
of L'Aquila, Via Vetoio 1 (Coppito), 67100 L'Aquila, Italy}
\email{valeria.iorio1@graduate.univaq.it}
\address[Markus Schmidtchen]{\newline Institute of Scientific Computing, Technische Universit\"at Dresden, Zellescher Weg 12-14, 01069
Dresden, Germany}
\email{markus.schmidtchen@tu-dresden.de}


\maketitle

\section*{Abstract}
We propose an upwind finite volume method for a system of two kinetic equations in one dimension that are coupled through nonlocal interaction terms. These cross-interaction systems were recently obtained as the mean-field limit of a second-order system of ordinary differential equations for two interacting species. Models of this kind are encountered in a myriad of contexts, for instance, to describe large systems of indistinguishable agents such as cell colonies, flocks of birds, schools of fish, herds of sheep. The finite volume method we propose is constructed to conserve mass and preserve positivity. Moreover, convex functionals of the discrete solution are controlled, which we use to show the convergence of the scheme. Finally, we investigate the scheme numerically.


\section{Introduction}
Models for collective behaviour have gained popularity in describing emergent phenomena in numerous fields such as social sciences ~\cite{toscani}, pedestrian flows~\cite{appertrolland}, traffic flow, ~\cite{klar_zanella}, and biology, see \cite{topaz_bertozzi_lewis, mogilner_eldestein, ccr}, and references therein. In particular, biological applications are often devoted to understanding the formation of patterns and self-organisation observed in nature, for instance, in swarms, schools of fish, and flocks of birds, cf.\ \cite{topaz_bertozzi}. 

The focus of this work is to present a finite volume scheme to study the following coupled two-species Vlasov-type system
\begin{align}
    \label{eq:main_system}
    \left\{
    \begin{array}{rl} 
    \ds \frac{\partial f}{\partial t}+v \frac{\partial f}{\partial x} &= \ds (\pii'\ast\rho +\pij' \ast \eta)\frac{\partial f}{\partial v}, \\[1em]
    \ds \frac{\partial g}{\partial t}+v \frac{\partial g}{\partial x} &= \ds ( \pjj' \ast\eta + \pji' \ast \rho)\frac{\partial g}{\partial v},
    \end{array}
    \right.
\end{align}
equipped with some non-negative initial datum $(f_0, g_0)\in L^1(\R \times \R)^2$, i.e.,
\[
	f(0,x,v)=f_0(x,v),\quad\text{and}\quad g(0,x,v)=g_0(x,v).
\]
Here, $(f,g)$ is a pair of phase-space densities describing the distribution of the two species on the domain $\brk{0,T}\times\R\times\R$. The potentials $\pii$ and $\pjj$ encode the self-interactions or intraspecific interactions, while $\pij$ and $\pji$ model the cross-interactions or interspecific interactions. Moreover,  $\rho$ and $\eta$ denote the associated macroscopic population densities, i.e., 
\[
    \rho (t,x)=\int_\R f(t,x,v)\dx{v}, \quad \text{and}\quad  \eta (t,x)=\int _\R g(t,x,v)\dx{v}.
\]
In \cite{dobrushin1979vlasov, neunzert2006introduction, braun1977vlasov}, the Vlasov equation was obtained as a many-particle limit. Recently, a lot of progress has been made for singular interactions, see \cite{hauray2007n, hauray2014particles, lazarovici2017mean} and the review article \cite{jabin2014review}. In the context of swarming particles (for instance, in the Cucker-Smale model \cite{cucker2007emergent} or the D'Orsogna-Chuang-Bertozzi-Chayes model \cite{d2006self}), Vlasov-type kinetic equations were obtained and studied in \cite{ccr, ha2008particle,carrillo2014derivation}. A good survey of second order particle systems in the context of a single interacting species can be found in \cite{carrillo2010particle} and references therein, and we refer to \cite{golse2016dynamics} for a great overview of mean-field limits in various contexts. More recently, the particle dynamics could be linked to macroscopic equations such as pressureless Euler-type equations, cf.\ \cite{carrillo2021mean}.  

Due to the transport nature of the equation, a popular way to treat the equation numerically is a semi-Lagrangian approach, which relies on splitting the operator into a transport and an interaction part, cf.\ \cite{crouseilles2009forward, charles2013enhanced, carrillo2014single} and references therein. Another important class of numerical approximations using the underlying transport structure are particle methods, cf.\ \cite{pinto2016uniform, pinto2014noiseless, cottet1984particle, SMAI-JCM_2018__4__121_0, wollman2000approximation}. Let us also refer to \cite{filbet_xiong, bessemoulin2022stability} for a spectral method, \cite{ayuso2009discontinuous, de2012discontinuous} discontinuous Galerkin method, and \cite{FILBET2001} for a type of flux-balanced method for Vlasov equations. A somewhat different approach to discretising the Vlasov-Poisson system consists of the class of finite-volume methods, cf.\ \cite{filbet} for a convergent scheme in one dimension.  For a survey of numerical approximations of the Vlasov equation, and kinetic equations in general,  we refer to \cite{filbet2003numerical} and \cite{dimarco2014numerical}, respectively. For a general survey of finite-volume methods we refer to \cite{barth2003finite, eymard2000finite}.

While models for one-species cases have been extensively studied, recently systems of multiple interacting species have received a lot of attention. In the case of overdamped Langevin dynamics, these nonlocal cross-interaction systems of partial differential equations were obtained by mean-field limits in \cite{difrancesco_fagioli} exploiting the gradient flow structure. In \cite{el2020finite}, the authors propose a finite volume method for a first order system of nonlocally interacting species proposed in \cite{difrancesco_fagioli}, and in \cite{carrillo_filbet_sch, carrillo_huang_sch} a finite volume method for nonlocally interacting species with size-exclusion effects via cross-diffusion was proposed. The derivation of its kinetic counterpart, i.e., \eqref{eq:main_system}, for two interaction species has been obtained as a many-particle limit, very recently, in \cite{CFI2022}, see also \cite{fetecau2015first}, for a one-species zero-inertia limit.

The existence theory for system \eqref{eq:main_system} is studied in arbitrary dimension in \cite{CFI2022}, where the authors prove the existence and uniqueness of a measure solution of the following $N$-species system
\[
	\partialt{f\i} + v\cdot \nabla _x f\i = \frac{1}{\varepsilon} \nabla_v \cdot \prt{vf\i} + \frac{1}{\varepsilon} \bigg( \sum_{j=1}^N \nabla K_{ij} \ast \rho\j\bigg)\cdot \nabla_v f\i,
\]
under some regularity assumptions on the interaction kernels. In the model above, where $f_i$ denotes the density of species $i$ as $i=1,\ldots,N$, and the parameter $\varepsilon>0$ represents the inertia, i.e., a small response time of individuals is considered.

In this paper, we propose a finite volume method for the two-species cross-interaction system \eqref{eq:main_system} and study its convergence. Subsequently, we investigate system \eqref{eq:main_system} numerically and study the numerical order of convergence.  Before constructing the numerical scheme and studying its properties, let us present some formal properties of solutions of cross-interaction system \eqref{eq:main_system} at the continuous level. 
For convenience, we shall, henceforth, use the notation
\[
    \Upsilon_f (t,x)\coloneqq \pii' \ast\rho + \pij' \ast\eta, \quad\text{and}\quad
    \Upsilon_g (t,x) \coloneqq \pjj' \ast\eta + \pji' \ast\rho.
\]

Now, we show that the solutions at the continuous level are bounded and positive. Indeed, for any smooth function $\phi\colon \R\to\R$, and $p\in \{f,g \}$, a straightforward computations shows
\begin{align*}
    \ddt \iint_{\R\times\R} \phi \circ p \dx x \dx v 
    &= \iint_{\R\times\R} \phi '(p)\partial _t p  \dxv \\ 
    &= \iint_{\R\times\R} \phi '(p) \prt*{ -v \frac{\partial p}{\partial x} + \Upsilon_p \frac{\partial p}{\partial v}} \dx x\dx v \\ 
    &= \iint_{\R\times\R} \begin{pmatrix} -v \\ \Upsilon_p \end{pmatrix} \cdot \nabla _{(x,v)} \phi (p) \dx x\dx v\\
    &=0.
\end{align*}
We introduce the notation
\[
	[x]^+ \coloneqq \max\{x, 0\}, \quad \text{and}\quad [x]^- \coloneqq - \min\{x, 0\}
\]
for the positive part and the negative part, respectively, of a real number $x$. Using $\phi(s) = [s]^-$, we observe that solutions to system \eqref{eq:main_system} remain non-negative if they were non-negative initially. Furthermore, if $\phi (s)=\brk{s- \norm{p_0}_{L^\infty}}^+$, with $p_0 \in \{f_0, g_0\}$, we obtain that the solution is bounded at the continuous level. Moreover, considering $\phi (s)=\abs{s}^q$, we see that the $L^q$-norms of the solution are controlled, a key ingredient mimicked on the numerical level in this paper.

\subsection{Structure of the paper}
This paper is organised as follows. In Section \ref{sec:num_meth}, we propose an upwind finite volume scheme for system \eqref{eq:main_system}. We introduce our notion of weak solutions to the system and present the statement of the main theorem. Section \ref{sec:Properties} is dedicated to proving that the numerical solutions remain non-negative and bounded under a suitable time step constraint. Hence, the approximated solution preserves the properties of the continuous solution. Section \ref{sec:Convergence} is devoted to the proof of the main result, the convergence of the numerical approximation to the weak solution of system \eqref{eq:main_system}. Finally, in Section \ref{sec:experiment}, we take a look at the convergence rates in numerical examples and compare them with the estimates obtained analytically.

\section{Derivation of the numerical method}
\label{sec:num_meth}
In this section, we shall derive a finite volume scheme to approximate the solutions to system \eqref{eq:main_system} on the domain $Q_T \coloneqq \prt{0,T}\times \prt{-L,L}\times\R$, equipped with periodic boundary conditions in physical space. Throughout this paper, we will use the following domains $Q\coloneqq \prt{-L,L}\times\R $, $\Omega_T \coloneqq (0,T)\times (-L,L)$ and $\Omega \coloneqq (-L,L)$.
\subsection{Numerical mesh}
We discretise the phase space by introducing cells
\[ 
	C_{i,j} = (x_{i-1/2},x_{i+1/2})\times (v_{j-1/2},v_{j+1/2}), 
\]
for $(i,j)\in \I\times \mathbb{Z}$, where $\I = \{0, \ldots, N_x - 1\}$. Here, $(x_{i-1/2})_{i\in \{0, \ldots, N_x\}}$ is a strictly increasing family of interfaces with $x_{-1/2} = -L$ and $x_{N_x-1/2} = L$. Similarly, $(v_{j-1/2})_{j\in \mathbb{Z}}$ denotes a strictly increasing sequence in $\R$, with $v_{j+1/2} \to \pm \infty$, as $j\to\pm \infty$.  We denote by \\
\begin{minipage}{.55\textwidth}
 $\Delta x_i=x_{i+1/2}-x_{i-1/2}$, for $i\in \I=\{0,\ldots,N_x-1\}$, the width of the spatial interval $(x\imh,x\ih)$. Additionally, we set $\Delta v_j=v_{j+1/2}-v_{j-1/2}$, for $j\in\mathbb{Z}$, to denote the width of the velocity interval $(v\jmh,v\jh)$. We associate with the mesh the parameter $h$ as the maximum of all space and velocity steps, i.e.,
	\[ 
	h \coloneqq \ds{ \max _{i\in \I,j\in \mathbb{Z}} \{ \Delta x_i, \Delta v_j \}>0.}
	\]
\end{minipage}\hspace*{.4cm}
\begin{minipage}{.40\textwidth}
	\vspace*{-.2cm}\begin{center}
		\scalebox{1.4}{
		\begin{tikzpicture}
		\coordinate (N) at (0,0);
		\coordinate (A) at (2.5,0);
		\coordinate (B) at (0,1.8);
		\coordinate (C)	at (-0.8,0);
		\coordinate (D) at (0,-0.8);
		
		\coordinate (A1) at (0.2,0);
		\coordinate (A1n) at (0.2,-0.1);
		\coordinate (A1p) at (0.2,0.1);
		\coordinate (A2) at (0.7,0);
		\coordinate (A2n) at (0.7,-0.1);
		\coordinate (A2p) at (0.7,0.1);
		\coordinate (A3) at (1.7,0);
		\coordinate (A3n) at (1.7,-0.1);
		\coordinate (A3p) at (1.7,0.1);
		\coordinate (A4) at (1.95,0);
		\coordinate (A4n) at (1.95,-0.1);
		\coordinate (A4p) at (1.95,0.1);
		
		\coordinate (B1) at (0,0.25);
		\coordinate (B1n) at (-0.1,0.25);
		\coordinate (B1p) at (0.1,0.25);
		\coordinate (B2) at (0,1);
		\coordinate (B2n) at (-0.1,1);
		\coordinate (B2p) at (0.1,1);
		\coordinate (B3) at (0,1.2);
		\coordinate (B3n) at (-0.1,1.2);
		\coordinate (B3p) at (0.1,1.2);
		\coordinate (B4) at (0,1.6);
		
		\coordinate (C1) at (-0.2,0);
		\coordinate (C1n) at (-0.2,-0.1);
		\coordinate (C1p) at (-0.2,0.1);
		\coordinate (C2) at (-0.45,0);
		
		\coordinate (D1) at (0,-0.5);
		\coordinate (D1n) at (-0.1,-0.5);
		\coordinate (D1p) at (0.1,-0.5);
		\coordinate (D2) at (0,-0.80);

		\coordinate (AB14) at (2.03,0.25);
		\coordinate (AB24) at (2.03,1);
		\coordinate (AB34) at (1.95,1.2);
		\coordinate (AB41) at (0.2,1.28);
		\coordinate (AB42) at (0.7,1.28);
		\coordinate (AB43) at (1.7,1.28);
		
		\coordinate (AD1) at (0.2,-0.62);
		\coordinate (AD12) at (0.7,-0.62);
		\coordinate (AD13) at (1.7,-0.62);
		\coordinate (AD14) at (1.95,-0.62);
		\coordinate (AD14b) at (2.03,-0.5);
		
		\coordinate (CB12) at (-0.45,0.25);			
		\coordinate (CB22) at (-0.45,1);
		\coordinate (CB31) at (-0.2,1.2);
		\coordinate (CB32) at (-0.53,1.2);
		\coordinate (CB32b) at (-0.45,1.28);
		
		\coordinate (CD1) at (-0.2,-0.62);
		\coordinate (CD12b) at (-0.53,-0.5);
		\coordinate (CD12) at (-0.45,-0.62);
		
		\draw[->] (N)--(A);
		\draw[->] (N)--(B);
		\draw[->] (N)--(C);
		\draw[->] (N)--(D);		
		
		\draw (A1n)--(A1p);	
		\draw (A2n)--(A2p);	
		\draw (A3n)--(A3p);	
		\draw (A4n)--(A4p);	
		
		\draw (B1n)--(B1p);	
		\draw (B2n)--(B2p);	
		\draw (B3n)--(B3p);	
		\draw (C1n)--(C1p);	
		\draw (D1n)--(D1p);	
		
		\draw[dotted] (A4)--(AB34);			
		\draw[dotted] (A1)--(AB41);
		\draw[dotted] (A3)--(AB43);
		\draw[dotted] (A2)--(AB42);
		
		\draw[dotted] (B1)--(AB14);
		\draw[dotted] (B2)--(AB24);
		\draw[dotted] (B3)--(AB34);
		
		\draw[dotted] (C1)--(CB31);	
		\draw[dotted] (C2)--(CB32b);		
		\draw[dotted] (B3)--(CB32);	
		
		\draw[dotted] (C1)--(CD1);	
		\draw[dotted] (C2)--(CD12);	
		\draw[dotted] (D1)--(CD12b);	
		
		\draw[dotted] (A4)--(AD14);			
		\draw[dotted] (A1)--(AD1);
		\draw[dotted] (A2)--(AD12);
		\draw[dotted] (A3)--(AD13);
		\draw[dotted] (D1)--(AD14b);
		
		\fill[black]  (A) circle [radius=0pt] node[right] {\footnotesize$x$}; 
		\fill[black]  (B) circle [radius=0pt] node[right] {\footnotesize$v$}; 
		\fill[black]  (N) circle [radius=1pt];
		\fill[black]  (0.1,0.05) node[below left] {\footnotesize\textbf{$0$}}; 
		\fill[black]  (1.2,0.3) circle [radius=0pt] node[above] {\footnotesize$C_{i,j}$}; 
		
		\fill[black]  (A2) circle [radius=0pt] node[below] {\tiny$x_{i-1/2}$}; 
		\fill[black]  (A3) circle [radius=0pt] node[below] {\tiny$x_{i+1/2}$}; 
		\fill[black]  (B1) circle [radius=0pt] node[left] {\tiny$v_{j-1/2}$}; 
		\fill[black]  (B2) circle [radius=0pt] node[left] {\tiny$v_{j+1/2}$}; 
		\end{tikzpicture}
	}
	\end{center}
\end{minipage} \\
We denote by $x\i$ (resp. $v\j$) the center of the interval $\prt{x\imh,x\ih}$ (resp. $\prt{v\jmh,v\jh})$. 

Additionally we call the mesh admissible if there exists $\alpha \in (0,1)$ such that
\[ 
	\alpha h \leq \Delta x_i, \Delta v_j \leq h,
\]
for all $(i,j)\in \I\times\Z$. Henceforth, we assume that our mesh admits the existence of such an $\alpha>0$. 

Finally, for some $N_T\in\N$, we set $\Delta t\coloneqq T/N_T$ for the time step and $t^n \coloneqq n\Delta t$, $n=0,\dots,N_T$, to denote the discrete time instances.

\subsection{Discretisation of the data}\label{sec:DefDiscrFunc}
We discretise the initial datum by a piecewise constant function.  We set
\[ 
	f^0_{i,j} \coloneqq \fint_{C_{i,j}} f_0(x,v)\,\dxv, \quad \text{and}\quad g^0_{i,j} \coloneqq \fint_{C_{i,j}} g_0(x,v)\,\dxv, 
\]
for $(i,j) \in \I\times \Z$ as the averaged integral $\fint$ of the initial datum $(f_0,g_0)$ over the cell $C_{ij}$.

To approximate the functions $f$ and $g$ we use piecewise constant functions on each cell  $(t\n,t\np)\times C\ij$, $n=0,\dots,N_T-1$, $(i,j) \in \I\times \Z$. For that purpose, we write these approximations as
\[
f_{i,j}^n \approx \fint_{C_{i,j}}f(t^n, x, v) \dxv, \quad\text{and}\quad g_{i,j}^n \approx \fint_{C_{i,j}}g(t^n, x, v)\dxv.
\]
Besides, we  define the piecewise constant approximations, $\rho_h$ and $\eta_h$, of the macroscopic densities $\rho$ and $\eta$ as 
\[
	\rho_h(t,x) = \rho_i^n, \qquad \text{and}\qquad \eta_h(t,x) = \eta_i^n,
\]
for $(t,x)\in [t^n,t^{n+1})\times (x_{i-1/2},x_{i+1/2})$, with $i\in \I$, and
\begin{align*}
	\rho_i^n \coloneqq \sum_{j\in \Z} \Delta v_j f_{i,j}^n, \quad \text{and} \quad \eta_i^n \coloneqq \sum_{j\in \Z} \Delta v_j g_{i,j}^n.
\end{align*}

However, these sums are over infinitely many entries $j\in\Z$. To implement the scheme, we have to work in a bounded domain. Therefore we need to truncate the velocity domain. Hence, we choose an arbitrary $v_h>0$ sufficiently large, such that $v_h\to \infty$ as $h\to 0$ and restrict the velocity domain to $(-v_h, v_h)$. We introduce the index set $\J \coloneqq \{ j\in\mathbb{Z} \, : \, \abs{v_{j+1/2}}\leq v_h \}$ which consists of all indices $j$ of the interfaces $(v_{j-1/2})_j$ that are inside the truncated velocity domain. Note that the choice of $v_h>0$ is made precise in Remark \ref{rem:ChoiceOfvh}.

Therefore, we define the piecewise constant approximation associated with the iterates obtained from the scheme, $(f_h,g_h)$ on $[0,T)\times [-L,L] \times(-v_h,v_h)$. They are extended by zero to the whole domain $[0,T)\times [-L,L] \times\R$, such that
\[
(f_h, g_h)(t,x,v) \coloneqq 
\begin{dcases} 
	(f_{i,j}^n, g_{i,j}^n), & \text{if $(t,x,v)\in [t^n,t^{n+1})\times C_{i,j}$ and $(i,j)\in \I\times \J,$} \\ 
	(0,0), & \text{else}. 
\end{dcases} 
\]

\subsection{Construction of the method}\label{sec:ConstMethod}
We obtain the finite volume approximation by integrating  system \eqref{eq:main_system} over a test cell, $(t\n,t\np)\times C\ij$ for a fixed $n\in\{0,\dots, N_T-1\}$, $i\in \I$ and $j\in\Z$. A formal computation yields
\begin{align*}
	\left\{
	\begin{array}{rl}
	\ds \fint_{C_{i,j}} \!\!f(t^{n+1}, x,v) - f(t^n, x,v)\,\dxv 
	\!\!\!\!& = \ds -\frac{\prescript{x}\!F_{i+1/2,j}^n-\prescript{x}\!F_{i-1/2,j}^n+\prescript{v}\!F_{i,j+1/2}^n-\prescript{v}\!F_{i,j-1/2}^n}{\abs{C_{i,j}}},\\[1em]
	\ds \fint_{C_{i,j}} \!\!g(t^{n+1}, x,v) - g(t^n, x,v)\,\dxv \!\!\!\!
	& = \ds -\frac{\prescript{x}\!G _{i+1/2,j}^n-\prescript{x}\!G_{i-1/2,j}^n+\prescript{v}\!G _{i,j+1/2}^n-\prescript{v}\!G_{i,j-1/2}^n}{\abs{C_{i,j}}},
	\end{array}
	\right.
\end{align*}
where $\prescript{x}\!F_{i+1/2,j}^n, \, \prescript{v}\!F_{i,j+1/2}^n$ are the fluxes of $f$ on the respective parts of the boundary of the cell $C_{i,j}$ given by
\[
	\begin{dcases}
	\prescript{x}\!F_{i+1/2,j}^n =\int_{t^n}^{t^{n+1}} \int_{v_{j-1/2}}^{v_{j+1/2}} v f(t,x_{i+1/2},v)\,\dx v\,\dx t,  \\
	\prescript{v}\!F_{i,j+1/2}^n =\int_{t^n}^{t^{n+1}} \int_{x_{i-1/2}}^{x_{i+1/2}} -\Upsilon_f(t,x) f(t,x,v_{j+1/2})\,\dx x\,\dx t,
	\end{dcases}
\]
for $(i,j)\in \I\times\Z$.
Similarly, $\prescript{x}\!G_{i+1/2,j}^n, \, \prescript{v}\!G_{i,j+1/2}^n$ are the fluxes of $g$ on the space and velocity boundary of the cell $C_{i,j}$, i.e., 
\[
	\begin{dcases}
	\prescript{x}\!G_{i+1/2,j}^n =\int_{t^n}^{t^{n+1}} \int_{v_{j-1/2}}^{v_{j+1/2}} v g(t,x_{i+1/2},v)\,\dx v\,\dx t, \\
	\prescript{v}\!G_{i,j+1/2}^n =\int_{t^n}^{t^{n+1}} \int_{x_{i-1/2}}^{x_{i+1/2}} -\Upsilon_g(t,x) g(t,x,v_{j+1/2})\,\dx x\,\dx t,
	\end{dcases}
\]
with $(i,j)\in\I\times\Z$.

If we apply the piecewise constant approximation for $f,g,\rho$, and $\eta$ as in Subsection \ref{sec:DefDiscrFunc}, we arrive at the discrete version of \eqref{eq:main_system}:
\begin{subequations}
\label{eq:finite-volume-method}
\begin{equation}
\label{eq:discrete_scheme}
\begin{dcases} 
f_{i,j}^{n+1}=f_{i,j}^n-\frac{1}{\abs{C_{i,j}}} \prt*{ \Fx_{i+1/2,j}^n- \Fx_{i-1/2,j}^n +\Fv _{i,j+1/2}^n -\Fv _{i,j-1/2}^n }, \\
g_{i,j}^{n+1}=g_{i,j}^n-\frac{1}{\abs{C_{i,j}}} \prt*{ \Gx _{i+1/2,j}^n- \Gx_{i-1/2,j}^n +\Gv _{i,j+1/2}^n -\Gv _{i,j-1/2}^n },
\end{dcases}
\end{equation}
for $(i,j)\in\I\times\J$ and $n\in \{0,\dots,N_T\}$. Note that we have replaced the continuous fluxes above by the discrete upwind fluxes $\Fx_{i+1/2,j}^n, \, \Fv_{i,j+1/2}^n, \, \Gx_{i+1/2,j}^n, \, \Gv_{i,j+1/2}^n$, defined as
\begin{equation}
	\label{eq:discr_one}
	\begin{dcases}
	\ds \Fx_{i+1/2,j}^n = \ds \Delta t \Delta v_j \,\prt*{f_{i,j}^n [v_j]^+  - f_{i+1,j}^n [v_j]^-}, \\[0.5em]
	\Fv_{i,j+1/2}^n = \Delta t \Delta x_i \, \prt*{f_{i,j}^n [(\Upsilon_f)_i^n]^- -f_{i,j+1}^n [(\Upsilon_f)_i^n]^+},
	\end{dcases}
\end{equation}
and, similarly,
\begin{equation}
	\label{eq:discr_two}
	\begin{dcases}
	\Gx_{i+1/2,j}^n =\Delta t \Delta v_j \prt*{g_{i,j}^n [v_j]^+  -g_{i+1,j}^n [v_j]^-}, \\[0.5em]
	\Gv_{i,j+1/2}^n = \Delta t \Delta x_i \prt*{g_{i,j}^n [(\Upsilon_g)_i^n]^- - g_{i,j+1}^n[(\Upsilon_g)_i^n]^+},
	\end{dcases}
\end{equation}
for $(i,j)\in \I\times\J$. The terms $(\Upsilon_f)_i^n$ and $(\Upsilon_g)_i^n$ are the approximations of the interaction terms $\Upsilon_f$ and $\Upsilon_g$ at the point $(t\n,x\i)$, and are defined by
\begin{align}
	\left\{
	\begin{array}{rl}
	\ds (\Upsilon_f)_i^n \! \! \! \!&\coloneqq \ds \sum_{k\in\I} \prt*{ \rho^n_k \int_{x_{k-1/2}}^{x_{k+1/2}}K_{11}'(x_i-y)\dx{y} +\eta^n_k \int_{x_{k-1/2}}^{x_{k+1/2}}K_{12}'(x_i-y) \dx{y}},\\[1.9em]
		\ds (\Upsilon_g)_i^n \! \! \! \!&\coloneqq \ds \sum_{k\in\I}\prt*{\eta^n_k  \int_{x_{k-1/2}}^{x_{k+1/2}}K_{22}'(x_i-y)\dx{y} + \rho^n_k \int_{x_{k-1/2}}^{x_{k+1/2}}K_{21}'(x_i-y)\dx{y}}.
	\end{array}
	\right.
\end{align}
\end{subequations}

The scheme is complemented with periodic boundary conditions in space, i.e.,
\begin{subequations}
	\label{eq:BDCond}
	\begin{equation}
		f_{N_x,j}^n=f_{0,j}^n, \quad g_{N_x,j}^n=g_{0,j}^n, 
	\end{equation}
	\begin{equation}
		f_{-1,j}^n=f_{N_x-1,j}^n, \quad g_{-1,j}^n=g_{N_x-1,j}^n,
	\end{equation}
where the values $f_{-1,j}^n$, $g_{-1,j}^n$, $f_{N_x,j}^n$, $g_{N_x,j}^n$ represent an approximation on a ``virtual cell''. In velocity we have no-flux boundaries, i.e.
	\begin{equation}
	\Fv _{i,j+1/2}^n=0=\Gv_{i,j+1/2}^n 
	\end{equation}
\end{subequations}
for all $(i,j)\in \I \times \mathbb{\Z}\setminus \J.$

\subsection{Finite volume scheme}\label{sec:FVM}
Throughout the paper we will use the following two representations of our scheme. First, we consider
\begin{equation*}
\tag{\ref{eq:discrete_scheme}}
\begin{dcases} 
f_{i,j}^{n+1}=f_{i,j}^n-\frac{1}{\abs{C_{i,j}}} \big( \Fx_{i+1/2,j}^n- \Fx_{i-1/2,j}^n +\Fv _{i,j+1/2}^n -\Fv _{i,j-1/2}^n \big), \\
g_{i,j}^{n+1}=g_{i,j}^n-\frac{1}{\abs{C_{i,j}}} \big( \Gx _{i+1/2,j}^n- \Gx_{i-1/2,j}^n +\Gv _{i,j+1/2}^n -\Gv _{i,j-1/2}^n \big),
\end{dcases}
\end{equation*}
for $n=0,\dots,N_T-1$ and $(i,j)\in \I\times\J$, where $\Fx$, $\Fv$, $\Gx$ and $\Gv$ are defined in \eqref{eq:discr_one} and \eqref{eq:discr_two}. Second, we can rewrite the scheme \eqref{eq:finite-volume-method}  and get by a short computation
\begin{align}
\label{eq:discrete_scheme_f_h_long}
\begin{aligned}
p_{i,j}^{n+1} =& \bigg( 1-\Delta t \bigg[ \frac{\abs{v_j}}{\Delta x_i}+\frac{\abs{(\Upsilon_p)_i^n}}{\Delta v_j} \bigg] \bigg) p_{i,j}^n +\Delta t \frac{[v_j]^-}{\Delta x_i} p_{i+1,j}^n + \Delta t \frac{[v_j]^+}{\Delta x_i} p_{i-1,j}^n \\& +\Delta t \frac{[(\Upsilon_p)_i^n]^+}{\Delta v_j} p_{i,j+1}^n + \Delta t \frac{[(\Upsilon_p)_i^n]^-}{\Delta v_j} p_{i,j-1}^n,
\end{aligned}
\end{align}
for $p\in\{f,g\}$ and $n=0,\dots,N_T-1$ and $(i,j)\in \I\times\J$. For both representations we use the boundary conditions \eqref{eq:BDCond}.

\subsection{Statement of the main result}
Before stating the main result,  let us introduce our notion of solutions.
\begin{definition}[Weak formulation] 
	\label{def:weak_sol}
	We call the pair $\prt{f,g}$ a weak solution to system \eqref{eq:main_system} if it satisfies
	\[
	\begin{dcases}
	\int_{Q_T}	f \bigg( \frac{\partial\varphi}{\partial t} + v \frac{\partial \varphi}{\partial x}-\Upsilon_f  \frac{\partial\varphi}{\partial v} \bigg) \dtxv + \int_Q f_0 \prt{x,v} \varphi \prt{0,x,v}\dxv = 0,\\
	\int_{Q_T}  g \bigg( \frac{\partial\varphi}{\partial t} + v \frac{\partial \varphi}{\partial x}-\Upsilon_g \frac{\partial\varphi}{\partial v} \bigg) \dtxv + \int_Q g_0 \prt{x,v} \varphi \prt{0,x,v}\dxv = 0,
	\end{dcases}
	\]
	for every $\varphi \in C_c^\infty([0,T)\times Q)$.
\end{definition}
 The main result of this paper is the following theorem.
\begin{reptheorem}{th:convergence}[Convergence of the scheme]
Let $p_0\in\{f_0,g_0\}$ be non-negative and bounded from above by a function $R$ of the form
\[ 
	R\prt{x,v}\coloneqq \frac{C}{1+\abs{v}\lone+\abs{x}\ltwo}, 
\]
with $\lambda_1 >1$, $\lambda_2\geq 1$ and $\lambda_2\leq \lambda_1$, for some $C>0$, $(x,v)\in Q$.
 Let $K_{ij} \in W^{2,\infty} \prt{\Omega}$, for $i,j \in \{1,2\}.$ Assume that there exists $\xi \in \prt{0,1}$ such that $\Delta t$ satisfies the condition
\begin{equation*}
	\Delta t\leq \frac{(1-\xi)\alpha}{v_h+\C}\, h,
\end{equation*}
for $\C \coloneqq \max \{ \norm{\pii'}_{L^\infty \prt{\Omega}} + \norm{\pij'}_{L^\infty \prt{\Omega}}, \norm{\pjj'}_{L^\infty \prt{\Omega}} + \norm{\pji'}_{L^\infty \prt{\Omega}}  \}$ and $v_h h^{1/2}\to 0$, as $h\to 0$. Denoting by $\prt{f_h, g_h}$ the numerical solution to the scheme \eqref{eq:finite-volume-method},  we have
\[
	f_h\rightharpoonup f, \qquad \text{and} \qquad g_h \rightharpoonup g,
\]
in $L^\infty \prt{Q_T}$ weakly-$\ast$ as $h\to 0,$ where $\prt{f,g}$ is a weak solution to system \eqref{eq:main_system}, in the sense of Definition \ref{def:weak_sol}.
\end{reptheorem}
This theorem we will prove in Section \ref{sec:Convergence}, after we have shown in Section \ref{sec:Properties} that the solution $\prt{f_h, g_h}$ has structure preserving properties.

\section{Properties of the numerical method and a priori estimates}\label{sec:Properties}
This section is dedicated to establishing the positivity and boundedness of the discrete approximation obtained in Subsection \ref{sec:FVM}.
\subsection{Positivity of the solution and CFL condition}
In order to mimic the structure-preserving properties of system \eqref{eq:main_system} at the level of the approximations, a stepsize restriction is required. Indeed, we assume that there exists $\xi \in (0,1)$ such that, for both species, $p\in\{f,g\}$, 
\begin{equation}
	\label{eq:clf}
	\frac{\Delta t}{\Cij} \prt*{\Delta v\j \abs{v\j}+\Delta x\i \abs{\prt{\Upsilon_p}\i\n}} \leq 1-\xi, 
\end{equation}
for all $\prt{i,j}\in \I\times \J$, and all $n\in \N.$

It is absolutely crucial to stress that, albeit apparently dependent on $n$, the stepsize restriction, \eqref{eq:clf}, can be shown to be satisfied uniformly in $n$. Indeed, we shall see in the subsequent proposition that it is independent of $n$ using a short induction argument.

\begin{prop}[Positivity preservation of the scheme]
	\label{prop:positivity}
	Let $K_{ij}\in W^{1,\infty} \prt{\Omega}$, $i,j\in \{1,2\}$, $p\in \{f,g\}$ with a non-negative initial condition $p_0\in \{f_0,g_0\}$ with $\|p_0\|_{L^1\prt{Q}}=1$. Assume that
	there exists $\xi \in \prt{0,1}$ such that the stepsize restriction
		\begin{equation}
		\label{eq:cflinitial}
		\frac{\Delta t}{\Cij} \prt{\Delta v\j \abs{v\j}+\Delta x\i \abs{\prt{\Upsilon_p}\i^0}} \leq 1-\xi, 
		\end{equation}
		is satisfied. Then, the following holds true:
	\begin{enumerate}
		\item $p\ij\n \geq 0$,  for all $(i,j)\in \I\times\J$, and $\norm{p_h(t\n)}_{L^1\prt{Q}}=\norm{p_h(t=0)}_{L^1\prt{Q}}$, for all $n\in\N$.
		\item If $\Delta t$ is chosen such that
		\begin{equation}
		\label{eq:deltat}
		\Delta t\leq \frac{(1-\xi)\alpha}{v_h+\C}\, h,
		\end{equation}
		with $\xi$ as in \eqref{eq:cflinitial} and $\C$ is defined by
		\[
		\C \coloneqq \max_{i\in \{1,2\}}  \sum_{j=1}^2 \norm{K_{ij}'}_{L^\infty \prt{\Omega}} ,
		\]
		then the CFL condition \eqref{eq:clf} is satisfied for the two species uniformly in $n\in\N$.
	\end{enumerate}
\end{prop}

\begin{rmk}
Note that, by the Proposition \ref{prop:positivity}, the positivity of $f_h$ and $g_h$ is guaranteed, and the scheme conserves the mass.
\end{rmk}

\begin{proof}
We proceed by induction. First, let us consider $n=0$. Since $p_0$ is non-negative we know that $p\ij^0\geq 0$, which implies \textit{(1)}. On the other hand, for $n=0$, the CFL condition \eqref{eq:clf} is satisfied by assumption. Next, let us assume for $n$ fixed that the statement \textit{(1)} and condition \eqref{eq:clf} hold true. Let us prove \textit{(1)} for $n+1$. We consider the representation \eqref{eq:discrete_scheme_f_h_long} of our scheme. Since, by induction assumption, $p\ij\n\geq 0$ for all $i\in \I$ and $j\in \J$, and  condition \eqref{eq:clf} is met for $n$, we derive from the representation \eqref{eq:discrete_scheme_f_h_long} that
\[ 
	p\ij\np\geq 0.
\]

Next, we prove the conservation of mass. Using the non-negativity in conjunction with the scheme, system \eqref{eq:discrete_scheme}, we compute
\begin{align*}
	\norm{p_h\prt{t\np}}_{L^1\prt{Q}}= & \sum_{i\in \I, j\in \J} \Cij\,  p\ij\np \\
 	=& \sum_{i\in \I, j\in \J} \Cij\,  p\ij\n - \sum_{i\in \I, j\in \J} \big( \Px_{i+1/2,j}^n- \Px_{i-1/2,j}^n\big)  \\
	&\qquad -\sum_{i\in \I, j\in \J} \big(\Pv _{i,j+1/2}^n -\Pv _{i,j-1/2}^n \big) \\
	=& \sum_{i\in \I, j\in \J} \Cij\,  p\ij\n,
\end{align*}
since both sums over the fluxes are telescopic sums and having exploited the periodic and no-flux boundary conditions. Therefore, we obtain
\[
	\norm{p_h\prt{t\np}}_{L^1\prt{Q}}= \sum_{i\in \I, j\in \J} \Cij\, p\ij\n =\norm{p_h\prt{t\n}}_{L^1\prt{Q}}=\norm{p_h\prt{0}}_{L^1\prt{Q}},
\]
where the last equality holds by assumption. Thus, the conservation of mass, \textit{(1)}, is guaranteed at the numerical level.
 	
Next, we prove statement \textit{(2)}. Let $\zeta_h\in\{\rho_h,\eta_h\}$ be the respective macroscopic density of  $p_h\in \{f_h,g_h\}$. We know that
\[
	\into \zeta_h \prt{t\np,x}\dx x=\sum_{i\in \I, j\in \J} \Cij\,  p\ij\np=\sum_{i\in \I, j\in \J} \Cij\, p\ij^0\leq\sum_{i\in \I, j\in \Z} \Cij p\ij^0 =1.
\]
Then we compute for $p_h =f_h$
\begin{align}
\label{eq:upsbound}
\begin{aligned}
	\abs{\prt{\Upsilon_f}_i^{n+1}} 
	&=\abss{\into \pii' \prt{x\i-y} \rho_h \prt{t\np,y}\,\dx{y} + \pij'\prt{x\i-y} \eta_h\prt{t\np,y}\dx{y} } \\
	&\leq \norm{\pii'}_{L^\infty\prt{\Omega}} \into \rho_h \prt{t\np,y}\dx{y}+\norm{\pij'}_{L^\infty\prt{\Omega}} \into \eta_h \prt{t\np,y}\dx{y}\\
	&\leq \C.
\end{aligned}
\end{align}
The same estimate can be established for the other species, $p_h=g_h$. Overall, this shows that
\[
	\Delta t \bigg( \frac{\abs{v\j}}{\Delta x\i} + \frac{\abs{\prt{\Upsilon_p}_i^{n+1}}}{\Delta v\j} \bigg) \leq\Delta t \bigg( \frac{v_h}{\alpha h}+\frac{\C}{\alpha h} \bigg), 
\]
for $p\in\{f,g\}.$ So, if we choose $\Delta t$ such that 
\[ \frac{\prt{1-\xi}\alpha}{v_h+\C} \,h >\Delta t, \]
we  can guarantee
\[
\Delta t \bigg( \frac{\abs{v\j}}{\Delta x\i} + \frac{\abs{\prt{\Upsilon_p}_i^{n+1}}}{\Delta v\j}  \bigg) \leq 1-\xi.
\]
Thus, the stepsize condition \eqref{eq:clf} is satisfied for both species and all $n\in \N$.
\end{proof}

\subsection{Boundedness of the solution and an a priori estimate}
We will begin by proving that the solutions of the scheme described in Subsection \ref{sec:FVM} are bounded in $L^p(Q)$ for each time $t\in(0,T)$. This we will prove using the next proposition.
\begin{prop} 
\label{prop:convex}
Consider a non-negative, convex function $\phi:\R\to\R$ such that
\[
\into \int_\R \phi (p_0(x,v)) \,\dxv < +\infty,
\] 
for $p_0 \in \{f_0, g_0\}$. Let the assumptions of Proposition~\ref{prop:positivity} hold true. Then, under the CFL condition \eqref{eq:clf}, the numerical solution satisfies
\begin{gather*}
\into \int_\R \phi \prt{p_h(t+\tau,x,v)}\,\dxv \leq \into \int_\R \phi (p_h(t,x,v))\,\dxv,
\end{gather*}
for $p\in \{f,g\}$ and every $t,\tau \geq 0.$ 
\end{prop}
\begin{proof}
Consider the representation \eqref{eq:discrete_scheme_f_h_long} of the discrete scheme. Under the CFL condition \eqref{eq:clf}, we observe that  $p\ij\np$ is a convex combination of $p\ij\n$, $p\ipj\n$, $p\imj\n$, $p\ijp\n$, $p\ijm\n$. By convexity of $\phi$, we obtain
\begin{align*}
\phi (p_{i,j}^{n+1} ) \leq & \bigg( 1-\Delta t \bigg[ \frac{\abs{v_j}}{\Delta x_i}+\frac{\abs{(\Upsilon_p)_i^n}}{\Delta v_j} \bigg] \bigg) \phi ( p_{i,j}^n ) +\Delta t \frac{[v_j]^-}{\Delta x_i} \phi (p_{i+1,j}^n) + \Delta t \frac{[v_j]^+}{\Delta x_i} \phi ( p_{i-1,j}^n ) \\
& +\Delta t \frac{[(\Upsilon_p)_i^n]^+}{\Delta v_j} \phi ( p_{i,j+1}^n ) + \Delta t \frac{[(\Upsilon_p)_i^n]^-}{\Delta v_j} \phi( p_{i,j-1}^n).
\end{align*}
Integrating in space and velocity, we have
\begin{align*}
	\into \int_\R \phi (p_h (t\np,x,v))\dxv =&\sum_{i\in \I, j\in \J} \Cij\phi (p\ij\np) \\
	\leq & \sum_{i\in \I} \sum_{j\in\J} \bigg[ \big( \Delta x\i \Delta v\j-\Delta t \prt{ \abs{v\j}\Delta v\j + \abs{\prt{\Upsilon_p}\i\n} \Delta x\i} \big) \phi \prt{p\ij\n} \\
	& \quad+ \Delta t \Delta v\j [v\j]^- \phi \prt{p\ipj\n}+\Delta t \Delta v\j [v\j]^+ \phi \prt{p\imj\n} \\
	& \quad+ \Delta t \Delta x\i \upsp^+ \phi \prt{p\ijp\n} + \Delta t \Delta x\i \upsp^- \phi \prt{p\ijm\n} \bigg].
\end{align*}
By shifting the indices and applying the boundary conditions \eqref{eq:BDCond} we get
\begin{equation}
\begin{aligned}
	\label{eq:phi-bd-proof}
	\into \int_\R \phi (p_h (t\np,x,v))\dxv \leq & \sum_{i\in \I, j\in \J} \Delta x\i\Delta v\j \phi (p\ij\n) \\ 
	= & \into \int_\R \phi (p_h(t\n,x,v))\dxv.
\end{aligned}
\end{equation}
Finally, let $t, \tau \geq 0$ be given. The statement follows from fixing integers, $n_0,n_1\in\N$ such that $t\in[t^{n_0}, t^{n_0+1})$ and $t+\tau \in [t^{n_1}, t^{n_1+1})$ and applying  estimate \eqref{eq:phi-bd-proof} iteratively. 
\end{proof}

\begin{rmk} Consequently, if the initial data satisfies $\|p_h(t=0)\|_{\infty}\leq  C $, for some constant $C>0$, we may consider  $\phi \prt{r}=\brk{r-C}^+$ to show that the numerical approximation, $p_h$, stays essentially bounded. Analogously, if $p_h(t=0)\in L^q(Q)$, using $\phi(r) = \abs{r}^q$, implies the uniform boundedness of $p_h(t)$ in $L^q(Q)$, for $p_h\in\{f_h,g_h\}$, a strategy similar to \cite{filbet, eymard2000finite}.
\end{rmk}

In the subsequent analysis, more refined bounds are required. To this end, we estimate the tails of $(f_h,g_h)$.
\begin{prop}\label{prop:Boundedness_TimeSpaceVelocity}
	Let the initial datum be non-negative and bounded from above by a function $R$ of the form
	\[ 
		R(x,v)= \frac{C}{1+\abs{v}\lone+\abs{x}\ltwo},
	\]
	for some $C>0$,  $\lambda_1 >1$, $\lambda_2 \geq 1,$ with $\lambda_2\leq \lambda_1,$ i.e., $0 \leq p_0(x,v)\leq R(x,v)$ with $p\in \{f,g\}$. 
	Then, there exists a constant $C_T >0$ depending on $\alpha$, $\lambda_1$, $\lambda_2$, $\C$ and the final time $T>0$ such that
\[
	0 \leq p_h(t,x,v)\leq C_T R_h(x,v), 
\]
for $(t,x,v)\in Q_T$, $p_h\in\{f_h,g_h\},$ where
\[ 
	R_h(x,v)\coloneqq \frac{C}{1+\abs{v\j}\lone+\abs{x_i}\ltwo}, 
\]
for $(x,v)\in C_{ij}.$ As a consequence, for $h$ small enough
\[ 
	0\leq \zeta_h (t,x)\leq C_T, 
\]
for $(t,x)\in \Omega _T$, where $\zeta_h\in\{ \rho_h, \eta_h\}$ are the respective macroscopic density of $p_h\in\{f_h,g_h\}.$
\end{prop}
\begin{proof}
Let $p_h\in \{f_h, g_h\}.$ By Proposition \ref{prop:positivity}, we know that $p_h$ is non-negative. Next, since $x\i=x\ip-\frac{1}{2}\prt{\Delta x\i +\Delta x\ip}$, setting $\Delta x\ih=\frac{1}{2}\prt{\Delta x\i+\Delta x\ip}$, by the definition of $R_h$ we have
\begin{align*}
\frac{R_h \prt{x\ip,v\j}}{R_h\prt{x\i,v_j}} & \leq \frac{1+\abs{v\j}\lone +\prt{\abs{x\ip}+\Delta x\ih}\ltwo}{1+\abs{v\j}\lone+\abs{x\ip}\ltwo} \\ 
& \leq \frac{1+\abs{v\j}\lone+\abs{x\ip}\ltwo + C\abs{x\ip}^{\lambda_2-1}\Delta x\ih+\mathcal{O}\prt{\prt{\Delta x\ih}^2}}{1+\abs{v\j}\lone+\abs{x\ip}\ltwo} \\
 & \leq 1 + C\frac{\abs{x\ip}^{\lambda_2-1}}{1+\abs{v\j}\lone+\abs{x\ip}\ltwo}\Delta x\ih +\mathcal{O}\prt{\prt{\Delta x\ih}^2}.
\end{align*}
In the same way, we obtain
\[ \frac{R_h \prt{x\im,v\j}}{R_h\prt{x\i,v\j}}\leq 1 + C\frac{\abs{x\i}^{\lambda_2-1}}{1+\abs{v\j}\lone+\abs{x\i}\ltwo}\Delta x\imh +\mathcal{O}\prt{\prt{\Delta x\imh}^2}.\]
Since, by assumption, $\lambda_2 \leq \lambda_1,$ we derive for $i\in \I$
\[ \frac{\abs{x\i}^{\lambda_2-1} \abs{v\j}}{1+\abs{v\j}\lone+\abs{x\i}\ltwo} \leq 1. \]
Indeed, if $\abs{v\j}\leq \abs{x\i}$, we get
\[ \abs{x\i}^{\lambda_2-1}\abs{v\j} \leq \abs{x\i}\ltwo \leq 1+\abs{v\j}\lone +\abs{x\i}\ltwo. \]
If, instead, $\abs{x\i} < \abs{v\j}$ and $\abs{v\j} \geq 1$, we obtain
\[ \abs{x\i}^{\lambda_2-1} \abs{v\j} \leq \abs{v\j}\ltwo \leq \abs{v\j}\lone \leq 1+\abs{v\j}\lone +\abs{x\i}\ltwo. \]
If, finally, $\abs{x\i} < \abs{v\j}$ and $\abs{v\j} < 1$, then
\[ \abs{x\i}^{\lambda_2-1}\abs{v\j} \leq \abs{x\i}^{\lambda_2-1} \leq 1 \leq 1+\abs{x\i}\ltwo+\abs{v\j}\lone. \]
Therefore, we derive that
\begin{align*}
[v_j]^-\frac{R_h \prt{ x\ip, v\j}}{R_h\prt{x\i,v_j}}+ [v_j]^+ \frac{R_h \prt{x\im,v\j}}{R_h\prt{x\i,v\j}}& \leq \abs{v\j}+ c_1\prt{\lambda_2}\Delta x\i \left( 2+\frac{\Delta x\ip+\Delta x\im}{\Delta x\i} \right) \\ 
& \leq \abs{v\j}+ c_1 \prt{\alpha,\lambda_2} \Delta x\i.
\end{align*}
Now, setting $\Delta v\jh=\frac{1}{2}\prt{\Delta v\j+\Delta v\jp}$, we have
\begin{align*}
\frac{R_h \prt{x\i,v\jp}}{R_h(x\i,v\j)} &= \frac{1+\abs{v\j}\lone+\abs{x\i}\ltwo}{1+\abs{v\jp}\lone+\abs{x\i}\ltwo} \leq \frac{1+\prt{\abs{v\jp}+ \Delta v\jh}\lone+\abs{x\i}\ltwo}{1+\abs{v\jp}\lone+\abs{x\i}\ltwo} \\
 &\leq 1+ C\frac{\abs{v\jp}^{\lambda_1-1}\Delta v\jh+\mathcal{O}\prt{\prt{\Delta v\jh}^2}}{1+\abs{v\jp}\lone+\abs{x\i}\ltwo} \\
 & \leq 1+C \frac{\abs{v\jp}^{\lambda_1-1}}{1+\abs{v\jp}\lone} \Delta v\jh +\mathcal{O}\prt{\prt{\Delta v\jh}^2} \\ 
 & \leq 1+c_2 \prt{\alpha,\lambda_1} \Delta v\j.
\end{align*}
In the same way, we obtain
\[ \frac{R_h \prt{x\i,v\jm}}{R_h\prt{x\i,v\j}}\leq 1+c_4\prt{\alpha, \lambda_1}\Delta v\j. \]
Set $c_0\prt{\alpha,\lambda_1, \lambda_2}=\max\{ c_1,c_2,c_3,c_4 \}.$
Set $A \coloneqq \prt{1+\Delta t c_0 \prt{1+\C}}$. We know that $p_0\prt{x,v}\leq A^0 R_h\prt{x,v}$. 
Let us proceed by induction. Assume $p_h\prt{t\n,x,v}\leq A^n R_h\prt{x,v}$. Using the numerical scheme \eqref{eq:discrete_scheme_f_h_long}, we have
\begin{align*}
\frac{p\ij\np}{R_h\prt{x\i,v\j}}  =& \bigg( 1-\Delta t \bigg[ \frac{\abs{v\j}}{\Delta x\i}+\frac{\abs{\prt{\Upsilon_p}\i\n}}{\Delta v\j}\bigg]\bigg)\frac{p\ij\n}{R_h\prt{x\i,v\j}} +\Delta t\frac{[v\j]^-}{\Delta x\i}\frac{p\ipj\n}{R_h\prt{x\ip,v\j}} \frac{R_h \prt{x\ip,v\j}}{R_h\prt{x\i,v\j}} \\ & +\Delta t \frac{[v\j]^+}{\Delta x\i} \frac{p\imj\n}{R_h \prt{x\im,v\j}} \frac{R_h \prt{x\im,v\j}}{R_h \prt{x\i,v\j}} +\Delta t \frac{\brk{\prt{\Upsilon_p}\i\n}^+}{\Delta v\j} \frac{p\ijp\n}{R_h \prt{x\i,v\jp}} \frac{R_h\prt{x\i,v\jp}}{R_h\prt{x\i,v\j}} \\ & + \Delta t \frac{\brk{\prt{\Upsilon_p}\i\n}^-}{\Delta v\j} \frac{p\ijm\n}{R_h \prt{x\i,v\jm}} \frac{R_h \prt{x\i,v\jm}}{R_h \prt{x\i,v\j}}.
\end{align*}
Hence, using the estimates above and \eqref{eq:upsbound}, we arrive at
\begin{align*}
\frac{p\ij\np}{R_h \prt{x\i,v\j}}  \leq & \bigg( 1- \Delta t \bigg[ \frac{\abs{v\j}}{\Delta x_i} +\frac{\abs{\prt{\Upsilon_p}\i\n}}{\Delta v\j} \bigg]\bigg) A\n + \Delta t \frac{\abs{v\j}}{\Delta x\i} A\n \bigg(1+c_0 \frac{\Delta x\i}{\abs{v\j}}\bigg) \\
& +\Delta t \frac{\abs{\prt{\Upsilon_p}\i\n}}{\Delta v\j} A\n \prt{1+c_0 \Delta v\j} \\ \leq & A\n \bigg( 1+ \Delta t c_0 \prt{1+\C}\bigg)= A\np.
\end{align*}
Thus, we obtain that for all $\prt{i,j}\in \I\times\mathbb{Z}$,
\[ \frac{p\ij\np}{R_h\prt{x\i,v\j}} \leq A\np. \]
By definition of $A^n$, we have for all  $n\in \{0,\ldots , \lceil T/\Delta t\rceil \}$ that $A\np < e^{c_0 \prt{1+\C} T}.$ Therefore, as in the continuous case, at the discrete level, there exists $C_T>0$ depending on $\alpha, \lambda_1, \lambda_2, \C, T$ such that
\[ p_h \prt{t,x,v}\leq C_T R_h\prt{x,v}, \] 
for $\prt{t,x,v}\in Q_T$.
Moreover, we have that
\begin{align*}
\int _\R R_h\prt{x,v}\,\dx v & = C \sum_{j\in\Z} \frac{\Delta v_j}{1+\abs{v\j}\lone+\abs{x\i}\ltwo} \leq 2C \sum_{j\in\N} \frac{h}{1+\prt{\alpha\brk{j-1} h}\lone+\abs{x\i}\ltwo} \\
& \leq\frac{2C}{\alpha} \sum_{j\in\N} \frac{\Delta v\jm}{1+\prt{\alpha \brk{j-1}h}\lone} \leq \frac{2C}{\alpha^{1+\lambda_1}}\int_0^\infty \frac{1}{1+v\lone}\dx v.
\end{align*}
Now, we have that
\begin{align*}
\int_0^\infty \frac{1}{1+v\lone}\dx v&= \int_0^1 \frac{1}{1+v\lone}\dx v + \int_1^\infty \frac{1}{1+v\lone}\dx v \\ & \leq 1+ \int_1^\infty \frac{1}{v\lone}\dx v = 1+ \frac{1}{\lambda_1-1}<\infty.
\end{align*}
Thus, we obtain that
\[ \zeta_h\prt{t,x} = \int_\R p_h \prt{t,x,v}\,\dx v \leq C_T \bigg( \frac{2C}{\alpha^{1+\lambda_1}} \int_0^\infty\frac{\dx v}{1+\abs{v}\lone} \dx v \bigg) < +\infty, \]
for $\zeta_h\in\{\rho_h, \eta_h\}.$
\end{proof}

\begin{rmk}
    \label{rem:ChoiceOfvh}
   With Proposition \ref{prop:Boundedness_TimeSpaceVelocity}, we can now choose an appropriate $v_h$ which is applied for the cut-off in the velocity domain in Section \ref{sec:DefDiscrFunc}. To this end, let $\varepsilon>0$. We want to choose $v_h>0$ such that
    \[ \int_{\R\backslash(-v_h,v_h)} p_h \prt{t,x,v} \dx v < \varepsilon, \]
     with $p_h \in \{f_h, g_h\}$ and $(t,x) \in \Omega_T$. Indeed, as in the proof of Proposition~\ref{prop:Boundedness_TimeSpaceVelocity}, we write 
     \begin{align*}
      \int_{\R\backslash(-v_h,v_h)} p_h \prt{t,x,v} \dx v &\leq  \int_{\R\backslash(-v_h,v_h)} C_T R_h\prt{x,v} \dx v
      \leq C_T \frac{2}{\alpha^{1+\lambda_1}} \int_{v_h}^\infty \frac{1}{1+v^\lambda_1}dv\\
      &\leq   \frac{2C_T}{\alpha^{1+\lambda_1}} \frac{1}{\lambda_1-1}v_h^{-\lambda_1+1},
     \end{align*}
for $t\in(0,T)$. Then we choose 
 \[
 v_h = \prt*{ \frac{2C_T}{\alpha^{1+\lambda_1}} \frac{1}{\lambda_1-1}\varepsilon^{-1}}^{\frac{1}{\lambda_1+1}}.
 \]
Such a choice of $v_h$ guarantees that the mass outside of $(-v_h,v_h)$ is less than $\varepsilon$ for all $(t,x)\in \Omega_T$. If $\varepsilon$  is much smaller than the machine epsilon, then the error that we are making by cutting off the functions $f$ and $g$ in the velocity domain is minimal with respect to the computational error. For the purpose of establishing the convergence result, Theorem \ref{th:convergence}, however, we need to impose that $v_h\to \infty$, as foreshadowed in the introduction. Indeed, the rate needed in our proof, cf.\ equation \eqref{eq:eh_final}, requires that $v_h h^{1/2} \to0$, as $h\to 0$. This requirement can be dropped, however, if the support of the solution is contained in a compact set, cf.\ Remark \ref{rem:error_estimate}. In the same vein, we may extend our result to unbounded spatial domains, as well, by imposing further restrictions on the rate at which $L\to \infty$.
\end{rmk}

\section{Convergence of the scheme} \label{sec:Convergence}
In this section, we will discuss the convergence behaviour of the scheme described in Section~\ref{sec:num_meth} and see that it converges to a weak solution of the original system \eqref{eq:main_system}. Before we prove the convergence of the scheme, we shall introduce an interpolation of the interaction terms that we will use in the proof.
\begin{definition}\label{def:UpsInter}	
	We define the interpolation of the interaction terms as 
	\begin{align*}
		\upsrhoh \prt{t,x} &:= \int_\Omega \pii' \prt{x-y} \rho_h \prt{t,y} \dx{y} + \int_\Omega \pij' \prt{x-y} \eta_h \prt{t,y}\dx{y}, \\
		\upsetah \prt{t,x} &:= \int_\Omega \pjj' \prt{x-y} \eta_h \prt{t,y} \dx{y} + \int_\Omega \pji' \prt{x-y} \rho_h \prt{t,y}\dx{y}.
	\end{align*}
\end{definition}

We are now in the position to prove Theorem \ref{th:convergence}, and we recall the statement here for the reader's convenience. 

\begin{theorem}[Convergence of the scheme] 
\label{th:convergence}
Let $p_0\in\{f_0,g_0\}$ be non-negative and bounded from above by a function $R$ of the form
\[ 
R\prt{x,v}\coloneqq \frac{C}{1+\abs{v}\lone+\abs{x}\ltwo}, 
\]
with $\lambda_1 >1$, $\lambda_2\geq 1$ and $\lambda_2\leq \lambda_1$, for some $C>0$, $(x,v)\in Q$.
Let $K_{ij} \in W^{2,\infty} \prt{\Omega}$, for $i,j \in \{1,2\}.$ Assume that there exists $\xi \in \prt{0,1}$ such that $\Delta t$ satisfies the condition
\begin{equation*}
	\Delta t\leq \frac{(1-\xi)\alpha}{v_h+\C}\, h,
\end{equation*}
for $\C \coloneqq \max \{ \norm{\pii'}_{L^\infty \prt{\Omega}} + \norm{\pij'}_{L^\infty \prt{\Omega}}, \norm{\pjj'}_{L^\infty \prt{\Omega}} + \norm{\pji'}_{L^\infty \prt{\Omega}}  \}$ and $v_h h^{1/2}\to 0$, as $h\to 0$.  Denoting by $\prt{f_h, g_h}$ the numerical solution to the scheme \eqref{eq:finite-volume-method},  we have
\[
f_h\rightharpoonup f, \qquad \text{and} \qquad g_h \rightharpoonup g,
\]
in $L^\infty \prt{Q_T}$ weakly-$\ast$ as $h\to 0,$ where $\prt{f,g}$ is a weak solution to system \eqref{eq:main_system}, in the sense of Definition \ref{def:weak_sol}.
\end{theorem} 
\begin{proof}
Let $p_h \in \{f_h, g_h \}$ and $\zeta_h \in \{\rho_h, \eta_h\}$ be the respective marcoscopic density. By Proposition \ref{prop:Boundedness_TimeSpaceVelocity} we know that $p_h$ is bounded in $L^\infty (Q_T)$, thus, by Banach-Alaoglu Theorem, we have that, up to a subsequence, there exists a function $p\in L^\infty (Q_T)$ such that
\[ 
	p_h\rightharpoonup p, 
\]
weakly-$\ast$ in $L^\infty\prt{Q_T}$, as $h\to 0$.
We also know that $\zeta_h$ is bounded in $L^\infty (\Omega_T)$, thus, up to a subsequence,
\[ \zeta_h\prt{t,x}\rightharpoonup\zeta \prt{t,x}, \]
weakly-$\ast$ in $L^\infty\prt{\Omega_T}$, as $h\to 0$.
Furthermore, we see that the density $\zeta (t,x)$ is equal to $\int_\R p\prt{t,x,v}\dx v$ a.e.. To prove this claim, let $\psi \in L^1(\Omega_T)$ and $\varepsilon>0$. By the uniform control from above by the function $R_h(x,v)$, we have
\begin{align}
	0 \leq p_h(t,x,v) \leq C_T R_h(x,v),
\end{align}
and therefore
\begin{align}
	\int_{\R\setminus(-M, M)} p_h(t,x,v)  \dx v < \varepsilon/(4 \|\psi\|_{L^1(\Omega_T)}),
\end{align}
whenever $h<H_1$, for some $H_1, M>0$ uniformly in $\Omega_T$, cf. Remark \ref{rem:ChoiceOfvh}. Next, by weak-* convergence, we know that there exists $H_2>0$ such that
\begin{align*}
	&\abs*{\int_{\Omega_T}	\prt*{\int_{-M}^M p_h(t,x,v) \dx v- \int_{-M}^M p(t,x,v) \dx v} \psi(x,t) \dx x \dx t}\\
	&=\abs*{\int_{\Omega_T}\int_{-M}^M	\prt*{p_h(t,x,v) -  p(t,x,v)} \dx v \psi(x,t) \dx x \dx t}\\
	&< \varepsilon/2,
\end{align*}
whenever $h< H_2$. Hence, for $h<\min(H_1, H_2)$, we have
\[
	\abs*{\int_{\Omega_T} \bigg( \zeta_h -\int_\R p\dx v \bigg) \psi \prt{t,x} \dx x \dx t } < \varepsilon,
\]
which proves the claim. Moreover, for $p=f$, for any $t\in(0,T)$, and $x\in \Omega$ fixed, we get
	\begin{align*}
		\prt{\prt{\Upsilon _f}_h - \Upsilon_f}(t,x) &= \int_\Omega \pii' (x-y)\rho_h (t,y)\dx{y} + \int_\Omega \pij' (x-y) \eta_h (t,y)\dx{y} \\
		& \qquad - \int_\Omega \pii' (x-y) \rho (t,y)\dx{y} - \int_\Omega \pij' (x-y) \eta (t,y) \dx{y} \\
		& \longrightarrow  \int_\Omega \pii' (x-y)\rho (t,y)\dx{y} + \int_\Omega \pij' (x-y) \eta (t,y)\dx{y} \\
		& \qquad - \int_\Omega \pii' (x-y) \rho (t,y)\dx{y} - \int_\Omega \pij' (x-y) \eta (t,y) \dx{y} \\
		&=\, 0,
	\end{align*}
	as $h\to 0$, since $\rho_h \rightharpoonup \rho$ and $\eta_h \rightharpoonup \eta$ weakly-$\ast$ in $L^\infty(\Omega)$.
	Consequently, we have pointwise convergence. Additionally, we have strong convergence in $L^1(\Omega_T)$ since
	\begin{align*}
		\norm{ (\Upsilon_f)_h}_{L^\infty(\Omega_T)} &\leq \C  (\norm{\rho_h}_{L^1(\Omega_T)} + \norm{\eta_h}_{L^1(\Omega_T)}) \leq C,
	\end{align*}
	with $C$ independent of $h$. Therefore we can apply Lebesgue's Dominated Convergence Theorem and get $\prt{\Upsilon_f}_h \to \Upsilon_f$ strongly in $L^1(\Omega_T)$. The same argumentation can be done for $p=g$.
	We derive that
	\[
	\int_{Q_T}\prt{\Upsilon_p}_h\, p_h \, \partialv{\varphi} \dxvt \to \int_{Q_T}  \Upsilon_p\, p\, \partialv{\varphi} \dxvt,
	\]
	as $h\to 0$, since $p_h \rightharpoonup p$ weakly-$\ast$ in $L^\infty(Q_T)$, and $\prt{\Upsilon_p}_h \to \Upsilon_p$ strongly in $L^1(\Omega_T)$. 
	Hence, for $\varphi\in C_c^\infty \prt{[0,T)\times Q}$, we see that the weak formulation in sense of Definition \ref{def:weak_sol} of the solution $p_h$ of the scheme converges to the weak formulation of the limit $p$ and
	\[
	\lim_{h\to 0} \int_{Q_T} p_h \bigg( \frac{\partial\varphi}{\partial t}+ v \frac{\partial \varphi}{\partial x} - \prt{\Upsilon_p}_h \frac{\partial \varphi}{\partial v} \bigg) \dtxv = \int_{Q_T} p \bigg( \frac{\partial\varphi}{\partial t}+ v \frac{\partial \varphi}{\partial x} - \Upsilon_p \frac{\partial \varphi}{\partial v} \bigg) \dtxv.
	\]
	It remains to show that the limit of the weak formulation of $p_h$ vanishes, that is, 
	\begin{equation}\label{wf:weakform_limit}
		\int_{Q_T} p_h \bigg( \frac{\partial\varphi}{\partial t}+ v \frac{\partial \varphi}{\partial x} - \prt{\Upsilon_p}_h \frac{\partial \varphi}{\partial v} \bigg) \dtxv + \int_Q p_0\prt{x,v} \varphi\prt{0,x,v}\dxv \to 0 
	\end{equation}
	for $\varphi\in C_c^\infty \prt{[0,T)\times Q}$, as $h\to 0$.
	We show this convergence by going back to the discrete scheme (\ref{eq:discrete_scheme_f_h_long}) and estimating the error between the discrete scheme and the weak formulation for $p_h$.

The following notation will be convenient:
\begin{align*}
	\mathcal{I}_t^h & \coloneqq  \int_{Q_T} p_h \prt{t,x,v} \frac{\partial \varphi}{\partial t} \prt{t,x,v} \dtxv +\int_Q p_0\prt{x,v} \varphi \prt{0,x,v} \dxv, \\
	\mathcal{I}_x^h & \coloneqq  \int_{Q_T} p_h \prt{t,x,v} v \frac{\partial \varphi}{\partial x} \prt{t,x,v} \dtxv, \\
	\mathcal{I}_v^h & \coloneqq  -\int_{Q_T} p_h \prt{t,x,v} \prt{\Upsilon_p}_h (t,x) \frac{\partial \varphi}{\partial v} \prt{t,x,v} \dtxv,
\end{align*}
where $\varphi \in C_c^\infty([0,T) \times Q)$ is arbitrary but fixed throughout. With the compactness from above, it is immediate to see that
\begin{align*}
	\lim_{h\to 0} I_{t}^h =  \int_{Q_T} p \prt{t,x,v} \frac{\partial \varphi}{\partial t} \prt{t,x,v} \dtxv +\int_Q p_0\prt{x,v} \varphi \prt{0,x,v} \dxv,
\end{align*}
as well as
\begin{align*}
	\lim_{h\to 0} I_{x}^h =  \int_{Q_T} p \prt{t,x,v} \, v \frac{\partial \varphi}{\partial x} \prt{t,x,v} \dtxv,
\end{align*}
and
\begin{align*}
	\lim_{h\to 0}\mathcal{I}_v^h = -\int_{Q_T} p \prt{t,x,v} \, \Upsilon_p (t,x) \frac{\partial \varphi}{\partial v} \prt{t,x,v} \dtxv.
\end{align*}
It remains to show \eqref{wf:weakform_limit}, i.e.,
\begin{align}
	\label{eq:conv-to-zero}
	\lim_{h\to 0} \mathcal I_t^h + \mathcal I_x^h + \mathcal I_v^h = 0.
\end{align}
In order to establish this limit, we exploit the discrete scheme (\ref{eq:discrete_scheme_f_h_long}). 
Let us observe that the scheme  \eqref{eq:discrete_scheme_f_h_long} can be rewritten as
\begin{align}
	\label{eq:disc-scheme-conv}
	\begin{split}
	\frac{p\ij\np - p\ij\n}{\Delta t} &= \frac{[v\j]^-}{\Delta x\i }(p\ipj\n - p\ij\n ) + \frac{[v\j]^+}{\Delta x\i }(p\imj\n - p\ij\n )\\[0.75em]
	&\qquad +\frac{\brk{\prt{\Upsilon_p}\i\n}^-}{\Delta v\j} (p\ijm\n - p\ij\n )
	+\frac{\brk{\prt{\Upsilon_p}\i\n}^+}{\Delta v\j} (p\ijp\n - p\ij\n ).
	\end{split}
\end{align}
Multiplying  \eqref{eq:disc-scheme-conv} by
\begin{align*}
	\varphi\ij\n \coloneqq \int_{C\ij\n} \varphi(t,x,v) \dx t \dx x \dx v,
\end{align*}
where $C\ij\n \coloneqq [t\n, t\np)\times C\ij$, and summing over $i \in \I$, $j\in\J$ and $n \in \{0, \ldots , N_T-1\}$, we obtain
\begin{align*}
	\mathcal J_t^h + \mathcal J_x^h + \mathcal J_v^h = 0,
\end{align*}
with
\begin{align*}
	\mathcal J_t^h & \coloneqq \sum_{n,i,j} 
		\frac{p\ij\np - p\ij\n}{\Delta t} \varphi\ij\n,\\
	\mathcal J_x^h &\coloneqq -\sum_{n,i,j} \bigg[
		\frac{[v\j]^-}{\Delta x\i }(p\ipj\n - p\ij\n )\varphi\ij\n  + \frac{[v\j]^+}{\Delta x\i }(p\imj\n - p\ij\n )\varphi\ij\n \bigg],\\
	\mathcal J_v^h &\coloneqq -\sum_{n,i,j} \bigg[ \frac{\brk{\prt{\Upsilon_p}\i\n}^-}{\Delta v\j} (p\ijm\n - p\ij\n ) \varphi\ij\n 
	+\frac{\brk{\prt{\Upsilon_p}\i\n}^+}{\Delta v\j} (p\ijp\n - p\ij\n )\varphi\ij\n \bigg].
\end{align*}
The strategy is to show that 
\[
	\abs*{\mathcal I_t^h + \mathcal J_t^h},\, 
	\abs*{\mathcal I_x^h + \mathcal J_x^h},\, 
	\abs*{\mathcal I_v^h + \mathcal J_v^h},
	\to 0,
\]
as  $h\to 0$, by enforcing condition \eqref{eq:clf}. This will imply \eqref{eq:conv-to-zero} and consequently our desired result \eqref{wf:weakform_limit}.  We proceed term by term.

\subsubsection*{Estimating $\mathcal J_t^h$}
We consider
\begin{align}
	\label{eq:Jth-defn}
	\begin{split}
	\mathcal J_t^h 
	&= \frac1{\Delta t}\sum_{i,j} \sum_{n=0}^{N_T -1} (p\ij\np - p\ij\n)\varphi\ij\n\\
	&= -\frac1{\Delta t}\sum_{i,j} \sum_{n=0}^{N_T -1}p\ij\np ( \varphi\ij^{n+1} -  \varphi\ij\n )  - \frac1{\Delta t}\sum_{i,j} p\ij^0\varphi\ij^0,
	\end{split}
\end{align}
having used integration by parts and the fact that $\varphi\ij^{N_T} = 0$ due to the compact support of $\varphi$. 
Using a Taylor expansion and the definition of $\varphi\ij\n$, we get
\begin{align}
	\label{eq:jij-final}
	\frac1{\Delta t} \prt{ \varphi\ij^{n+1}-\varphi\ij\n} = \int_{C\ij^{n+1}} \bigg[ \frac{\partial \varphi}{\partial t}(t,x,v)  \bigg] \dtxv + \mathcal O(\Delta t^2) \Cij.
\end{align}
Substituting \eqref{eq:jij-final} into \eqref{eq:Jth-defn} yields
\begin{align*}
	\mathcal J_t^h
		&= \sum_{i,j} \bigg\{-  \sum_{n=0}^{N_T-1} \bigg[ \int_{C\ij\np} p\ij\np  \frac{\partial \varphi}{\partial t} \dtxv + |C_{ij}^{n+1}| p\ij\np  \mathcal O(\Delta t)\bigg] -\frac1{\Delta t} p\ij^0 \varphi\ij^0 \bigg\} \\
		&=\sum_{i,j} \bigg\{- \sum_{n=1}^{N_T-1} \int_{C\ij\n} p\ij\n \frac{\partial \varphi}{\partial t} \dtxv - \frac{1}{\Delta t} p\ij^0 \varphi\ij^0\bigg\} + \mathcal O (\Delta t),
\end{align*}
since the test function has compact support and having used the boundedness of the $L^1(Q_T)$-norm. Rearranging the expression of $\mathcal{I}_t^h$, we find
\begin{align*}
	\mathcal J_t^h + \mathcal I_t^h 
	&= \sum\ij \bigg[ -\frac{1}{\Delta t} p\ij^0\varphi\ij^0 + p\ij^0 \int_{C\ij} \int_0^{\Delta t} \frac{\partial \varphi}{\partial t} \dtxv \\
	& \qquad \qquad \qquad \qquad \qquad + \int_{C\ij} p_0 (x,v) \varphi (0,x,v)\dxv  \bigg] + \mathcal O(\Delta t)\\
	&= \sum\ij \bigg[ - \int_{C\ij}p\ij^0 \varphi(0, x, v)\dx x\dx v 
	+ \int_{C\ij} p_0 (x,v) \varphi (0,x,v)\dxv  \bigg] + \mathcal O(\Delta t),\\
\end{align*}
having used a Taylor expansion in time of the test function, $\varphi\ij^{0}$. Thus, we obtain that
\[ \abs{ \mathcal J_t^h + \mathcal I_t^h } \leq C \norm{p_0 -p_h(0)}_{L^1 (Q)} + C\Delta t \to 0, \]
as $h \to 0$.

\subsubsection*{Estimating $\mathcal J_x^h$} Next, let us consider 
\begin{align*}
	\mathcal J_x^h &= \sum_{n,i,j}
		\frac{[v\j]^-}{\Delta x\i }(p\ij\n - p\ipj\n)\varphi\ij\n  + \frac{[v\j]^+}{\Delta x\i }(p\ij\n - p\imj\n)\varphi\ij\n\\
		&= \mathcal J_x^{h,-} + \mathcal J_x^{h,+},
\end{align*}
with 
\begin{equation}
\label{eq:Jxh}
	\mathcal J_x^{h,+} \coloneqq \sum_{n,i,j}   \frac{[v\j]^+}{\Delta x\i } \prt{p\ij\n - p\imj\n} \varphi\ij\n, \quad \text{and}\quad
	\mathcal J_x^{h,-} \coloneqq \sum_{n,i,j}  \frac{[v\j]^-}{\Delta x\i } \prt{p\ij\n - p\ipj\n}  \varphi\ij\n.
\end{equation}
Considering a Taylor expansion of the test function, we have
\[ 
	\intx \varphi\prt{t,x,v}\dx x =\varphi \prt{t,x_{i\pm 1/2},v}\Delta x\i + \mathcal{O}\prt{\Delta x\i^2}. 
\]
We begin by addressing $	\mathcal J_x^{h,+}$. Inserting the Taylor expansion yields
\begin{align}
	\label{eq:Jxplus}
	\mathcal J_x^{h,+} 
	=& \sum_{n,j} \sum_{i=0}^{N_x-1} [v\j]^+ \prt{p\ij\n-p\imj\n} \int_{t\n}^{t\np} \int_{v\jmh}^{v\jh} \varphi\prt{t,x\imh,v}\dx v\dx t+\mathcal E_x^{h,+},
\end{align}
where
\begin{align}
	\label{eq:err-x-plus}
	\abs*{\mathcal E_x^{h,+}} \leq C \sum_{n,i,j} \Cijn [v\j]^+  \abs{p\ij\n-p\imj\n}.
\end{align}
Upon integrating by parts and using the spatial boundary conditions, we get
\begin{align*}
	&\sum_{n,j} [v\j]^+ \bigg[ p_{N_x-1,j}\n   \intt\intv \varphi \prt{t,x_{N_x-3/2},v} \dx v \dx t\\
	&\qquad\qquad  - p_{-1,j}\n \intt\intv \varphi\prt{t,x_{-3/2},v}\dx v\dx t \\ 
	&\qquad\qquad - \sum_{i=-1}^{N_x-2} p\ij\n \intt \intv \prt{ \varphi\prt{t,x\ih,v}-\varphi \prt{t,x\imh,v}} \dx v\dx t\bigg] \\[0.75em]
	&= -\sum_{n,j}\sum_{i=0}^{N_x-1} [v\j]^+ p\ij\n \intt \intv \prt{ \varphi \prt{t,x\ih,v}-\varphi \prt{t,x\imh,v}}\dx v\dx t\\
	&= -\sum_{n,i,j} [v\j]^+ \int_{C\ij\n} p_h(t,x,v) \partialx \varphi(t,x,v) \dx t \dx x \dx v,
\end{align*}
having used the fact that $\varphi$ is compactly supported. Substituting this into \eqref{eq:Jxplus}, we get
\begin{align}
	\label{eq:Jxplus-final}
	\mathcal J_x^{h,+} = -\sum_{n,i,j}[v\j]^+ \int_{C\ij\n} p_h(t,x,v) \partialx \varphi(t,x,v) \dx t \dx x \dx v + \mathcal E_x^{h,+}.
\end{align}
Next, let us address $\mathcal J_x^{h,-}$. We have
\[
	\mathcal J_x^{h,-} 
	= -\sum_{n,j} \sum_{i=0}^{N_x-1} [v\j]^- \prt{p\ipj\n-p\ij\n} \intt\intv \varphi\prt{t,x\ih,v}\dx v\dx t + \mathcal E_x^{h,-},
\]
where 
\begin{align}
	\label{eq:err-x-minus}
	\abs*{\mathcal E_x^{h,-}}\leq C \sum_{n,i,j} \Cijn [v]^- \abs*{p\ipj\n - p\ij\n}.
\end{align}
Estimating the first term of $\mathcal J_x^{h,-}$ in a fashion similar to the one above, we obtain
\begin{align}
	\label{eq:Jxminus-final}
	\mathcal J_x^{h,-} = \sum_{n,i,j} [v\j]^- \int_{C\ij\n} p_{h}(t,x,v) \partialx \varphi(t,x,v) \dx x \dx v \dx t + \mathcal E_x^{h,-}.
\end{align}
Summing \eqref{eq:Jxplus-final} and \eqref{eq:Jxminus-final}, we derive
\begin{align*}
	\mathcal J_x^h &= -\sum_{n,i,j} \int_{C\ij\n}  p_{h}(t,x,v)\, v\j \partialx \varphi(t,x,v) \dx x \dx v \dx t + \mathcal E_x^{h,+} + \mathcal E_x^{h,-}\\
	&=-\sum_{n,i,j} \int_{C\ij\n} p_{h}(t,x,v)\, v \partialx \varphi(t,x,v) \dx x \dx v \dx t + \mathcal E_x^{h,+} + \mathcal E_x^{h,-} + \mathcal O(h).
\end{align*}
Thus, we may conclude our estimate by summarising
\[
	\mathcal J_x^h + \mathcal I_x^h = \mathcal E_x^{h,+} + \mathcal E_x^{h,-} + \mathcal O(h).
\]

\subsubsection*{Estimating $\mathcal J_v^h$}
We consider
\begin{align*}
	\mathcal J_v^h 
	&= \sum_{n,i,j}\frac{\brk{\prt{\Upsilon_p}\i\n}^-}{\Delta v\j} (p\ij\n - p\ijm\n) \varphi\ij\n 
	+\frac{\brk{\prt{\Upsilon_p}\i\n}^+}{\Delta v\j} (p\ij\n - p\ijp\n)\varphi\ij\n\\
	&= \mathcal J_v^{h,-} + \mathcal J_v^{h,+},
\end{align*}
where
\begin{align*}
	\mathcal J_v^{h,+} \coloneqq \sum_{n,i,j}\frac{\brk{\prt{\Upsilon_p}\i\n}^+}{\Delta v\j} (p\ij\n - p\ijp\n)\varphi\ij\n, \quad \text{and} \quad \mathcal J_v^{h,-} \coloneqq \sum_{n,i,j}\frac{\brk{\prt{\Upsilon_p}\i\n}^-}{\Delta v\j} (p\ij\n - p\ijm\n) \varphi\ij\n.
\end{align*}
Again, we proceed by Taylor expanding the test function, i.e.,
\begin{align*}
	\intv \varphi \prt{t,x,v} \dx v = \varphi \prt{t,x,v_{j\pm 1/2}}\Delta v\j + \mathcal{O}\prt{\Delta v\j^2}. 
\end{align*}
Now, let $J\in \N$ such that $\mathrm{supp}(\varphi(t,x, \cdot)) \subset (v_{-J-1/2}, v_{J+1/2}) $. Then we have
\begin{align*}
	\mathcal J_v^{h,-} 
	&=\sum_{n,i}\sum_{j=-J}^{J} \frac{\upsp^-}{\Delta v\j} \prt{p\ij\n-p\ijm\n}  \varphi\ij\n \\ 
	&= \sum_{n,i} \sum_{j=-J}^{J} \upsp^- \prt{p\ij\n-p\ijm\n} \intt \intx \varphi \prt{t,x,v\jmh}\dx x\dx t + \mathcal E_v^{h,-},
\end{align*}
where
\begin{align}
	\label{eq:err-v-minus}
	\abs{\mathcal E_v^{h,-}} \leq C \sum_{n,i,j} \Cijn \brk{\prt{\Upsilon_p}\i\n}^- \abs*{p\ij\n - p\ijm\n}.
\end{align}
By manipulating the first term in $\mathcal J_v^{h,-}$, we get
\begin{align*}
	&\sum_{n,i}\upsp^- \bigg[ p_{i,J}\n \intt \intx \varphi \prt{t,x,v_{J-1/2}}\dx x\dx t \\
	&\qquad\qquad\qquad - p_{i,-J-1}\n \intt \intx \varphi \prt{t,x,v_{-J-3/2}}\dx x\dx t \\
	&\qquad\qquad\qquad  -\sum_{j=-J-1}^{J-1} p\ij\n \intt\intx \prt{\varphi \prt{t,x,v\jh}-\varphi\prt{t,x,v\jmh}}\dx x\dx t \\
	&= -\sum_{n,i} \upsp^- \sum_{j=-J}^{J} p\ij\n \intt \intx \prt{\varphi\prt{t,x,v\jh}-\varphi \prt{t,x,v\jmh}} \dx x\dx t\\
	&= -\sum_{n,i,j} \int_{C\ij\n } p_h(t,x,v) \upsp^-  \partialv{\varphi} (t,x,v)\dx v\dx x\dx t,
\end{align*}
having used the compact support of the test function $\varphi$, as well as the boundary conditions. Next, let us consider
\begin{align*}
	\mathcal J_v^{h,+} =& -\sum_{n,i,j} \frac{\upsp^+}{\Delta v\j} \prt{p\ijp\n-p\ij\n} \varphi\ij\n\\
	=& -\sum_{n,i} \sum_{j=-J}^{J} \upsp^+ \prt{p\ijp\n-p\ij\n} \intt \intx \varphi \prt{t,x,v\jh}\dx x\dx t + \mathcal E_v^{h,+},
\end{align*}
where
\begin{align}
	\label{eq:err-v-plus}
	\abs{\mathcal E_v^{h,+}} \leq C \sum_{n,i,j} \Cijn \upsp^+ \abs*{p\ijp\n - p\ij\n}.
\end{align}
Treating the first term of $\mathcal J_v^{h,+}$ in a way parallel to the one above, we directly get
\begin{align*}
	\mathcal J_v^{h,+} = \sum_{n,i,j} \int_{C\ij\n}  p_h(t,x,v)  \upsp^+ \partialv{\varphi}(t,x,v) \dx v \dx x\dx t + \mathcal E_v^{h,+}.
\end{align*}
In conclusion, we combine both terms of  $\mathcal J_v^{h}$ to get
\begin{align*}
	\mathcal J_v^{h} 
	&= \mathcal J_v^{h,-} + \mathcal J_v^{h,-} \\
	&= \sum_{n,i,j} \int_{C\ij\n}  p_h(t,x,v) \,  (\Upsilon_p)_i^n  \partialv{\varphi}(t,x,v) \dx v \dx x\dx t + \mathcal E_v^{h,-} + \mathcal E_v^{h,+}\\
	&= \sum_{n,i,j} \int_{C\ij\n}  p_h(t,x,v)\,  \prt{\Upsilon_p}_h  \partialv{\varphi}(t,x,v) \dx v \dx x\dx t + \mathcal E_v^{h,-}  +  \mathcal E_v^{h,+} +  \mathcal O(h),
\end{align*}
having used
\begin{align*}
	 &\abs*{\sum_{n,i,j} \int_{C\ij\n}  p_h(t,x,v)  \brk{(\Upsilon_p)_i^n - \prt{\Upsilon_p}_h} \partialv{\varphi}(t,x,v) \dx v \dx x\dx t}\\
	 &\leq C \sum_{n,i,j} \Cijn p\ij\n \abs{(\Upsilon_p)_i^n - \prt{\Upsilon_p}_h}\\
	 &\leq C h,
\end{align*}
since $K\ij\in W^{2,\infty}(\Omega)$. Thus
\begin{align*}
	\mathcal I_v^h + \mathcal J_v^h   = \mathcal E_v^{h,+} +  \mathcal E_v^{h,-}  +  \mathcal O(h).
\end{align*}

\subsubsection*{Combination of all estimates}
Next, we combine all estimates from above and define
\begin{align*}
	e_h & \coloneqq \abs*{(\mathcal J_t^h + \mathcal J_x^h + \mathcal J_v^h) + (\mathcal I_t^h + \mathcal I_x^h + \mathcal I_v^h)}\\
	&\leq \abs*{\mathcal J_t^h + \mathcal I_t^h } + \abs*{\mathcal J_x^h + \mathcal I_x^h} + \abs*{\mathcal J_v^h + \mathcal I_v^h }\\
	&\leq
	\tilde e_h + C\prt*{h + \Delta t + \norm{p_h(0)- p_0}_{L^1(Q)}},
\end{align*}
with 
\[
	\tilde e_h := \abs{\mathcal E_x^{h,+}} + \abs{\mathcal E_x^{h,-}} + 
	\abs{\mathcal E_v^{h,+}} + \abs{\mathcal E_v^{h,-}}.
\]
Using equations \eqref{eq:err-x-plus}, \eqref{eq:err-x-minus}, \eqref{eq:err-v-minus}, \eqref{eq:err-v-plus}, we obtain
\begin{align*}
	\tilde e_h 
	&\leq \sum_{n,i,j} \Cijn \bigg[[v\j]^+  \abs{p\ij\n-p\imj\n} + [v\j]^- \abs*{p\ipj\n - p\ij\n} \\
	&\qquad\qquad\qquad + \brk{\prt{\Upsilon_p}\i\n}^- \abs*{p\ij\n - p\ijm\n} + \upsp^+ \abs*{p\ijp\n - p\ij\n}\bigg]\\
	&\leq h \Delta t \sum_{n,i,j}  \bigg[\Delta v\j [v\j]^+  \abs{p\ij\n-p\imj\n} + \Delta v\j [v\j]^- \abs*{p\ipj\n - p\ij\n} \\
	&\qquad\qquad\qquad + \Delta x\i \brk{\prt{\Upsilon_p}\i\n}^- \abs*{p\ij\n - p\ijm\n} + \Delta x\i \upsp^+ \abs*{p\ijp\n - p\ij\n}\bigg]\\
	&\leq h \Delta t \bigg[ \sum_{n,i,j} \Delta v\j \abs{v\j} + \Delta x\i \abs{(\Upsilon_p)_i^n} \bigg]^{1/2} \\
	&\qquad \quad \times \bigg[ \sum_{n,i,j} \Delta v\j [v\j]^+ \brk{p\ij\n-p\imj\n}^2 + \Delta v\j [v\j]^- \brk{p\ij\n-p\ipj\n}^2 \\
 & \qquad \qquad \qquad + \Delta x\i \upsp ^+ \brk{ p\ij\n-p\ijp\n}^2 + \Delta x\i \upsp ^- \brk{ p\ij\n - p\ijm\n}^2 \bigg]^{1/2}.
\end{align*}
Now, by using the Cauchy-Schwarz inequality, we obtain
\begin{align*}
	\prt*{\sum_{n,i,j} \prt*{ \Delta v\j \abs{v\j} + \Delta x\i \abs{(\Upsilon_p)_i^n} }}^{1/2} 
	&\leq (v_h^{1/2} + \C^{1/2}) \prt*{\sum_{n,i,j} (\Delta v\j + \Delta x\i)}^{1/2} \\
	&\leq  (v_h^{1/2} + \C^{1/2}) \prt*{\frac{8LT}{\alpha}}^{1/2}\prt*{\frac{v_h}{h \Delta t}}^{1/2} ,
\end{align*}
and thus we have
\begin{align}
	\label{eq:conv-estimate-interm}
	\tilde e_h 
	&\leq  C \prt*{Cv_h^{1/2} + v_h} \Delta t^{1/2} h^{1/2} \mathcal R^{1/2},
\end{align}
where
\begin{align*}
	\mathcal R 
	&\coloneqq \sum_{n,i,j} \Delta v\j [v\j]^+ \abs*{p\imj\n - p\ij\n}^2 + \Delta v\j [v\j]^- \abs*{p\ipj\n - p\ij\n}^2 \\
	&\quad \qquad + \Delta x\i \upsp ^+ \abs*{p\ijp\n - p\ij\n}^2 + \Delta x\i \upsp ^- \abs*{p\ijm\n -  p\ij\n}^2.
\end{align*}
Using the fact that
\begin{align*}
	\abs*{\hat p - p\ij\n}^2 = 2 \prt*{p\ij\n - \hat p} p\ij\n + \hat p^2 - \abs*{p\ij\n}^2, 
\end{align*}
in particular for $\hat p\in\set{p_{i\pm 1, j}^n, p_{i, j\pm 1}}$, we may rewrite $\mathcal R$ such that
\begin{align*}
	\mathcal R 
	&= 2 \sum_{n,i,j} \bigg[ \Delta v\j [v\j]^+ \brk{p\ij\n-p\imj\n} p\ij\n + \Delta v\j [v\j]^- \brk{p\ij\n-p\ipj\n} p\ij\n \\ 
	& \qquad \qquad + \Delta x\i \upsp^- \brk{p\ij\n-p\ijm\n} p\ij\n + \Delta x\i \upsp^+ \brk{p\ij\n-p\ijp\n} p\ij\n \bigg] \\[0.75em]
	& \quad + \sum_{n,i,j} \bigg[ \Delta v\j [v\j]^+ \prt*{ \abs{p\imj\n}^2 - \abs{p\ij\n}^2 } + \Delta v\j [v\j]^- \prt*{ \abs{p\ipj\n}^2 - \abs{p\ij\n}^2 }\\ 
	& \qquad + \Delta x\i \upsp^- \prt*{\abs{p\ijm\n}^2 - \abs{p\ij\n}^2} + \Delta x\i \upsp^+ \prt{ \abs*{p\ijp\n}^2 - \abs{p\ij\n}^2 }\bigg]. 
\end{align*}
We observe that the last summation contains telescopic sums such that, indeed,
\begin{align*}
	\mathcal R 
	&\leq 2 \sum_{n,i,j} p\ij\n \bigg[ \Delta v\j [v\j]^+ \brk{p\ij\n-p\imj\n} + \Delta v\j [v\j]^- \brk{p\ij\n-p\ipj\n} \\ 
	& \qquad \qquad \qquad + \Delta x\i \upsp^- \brk{p\ij\n-p\ijm\n}+ \Delta x\i \upsp^+ \brk{p\ij\n-p\ijp\n} \bigg] \\[0.75em]
	& \qquad + \sum_{n,i} \bigg[ \Delta x\i \upsp ^+ \prt{\abs{p\n_{i,-J}}^2-\abs{p\n_{i,J+1}}^2} + \Delta x\i \upsp ^- \prt{ \abs{p\n_{i,-J-1}}^2 - \abs{p\n_{i,J}}^2 } \bigg],
\end{align*}
where we factored out a $p\ij\n$ in the first term. Using the scheme \eqref{eq:disc-scheme-conv} we see that 
\begin{align*}
	\mathcal R \leq 2 \sum_{n,i,j} \Cij p\ij\n \frac{p\ij\np - p\ij\n}{\Delta t} + \frac{C}{\Delta t},
\end{align*} 
where the last term comes from bounding the boundary terms, i.e., the second sum in the previous equation. By convexity of $s\mapsto s^2$, we can estimate further to get
\begin{align}
	\label{eq:R-final}
	\mathcal R \leq \sum_{n,i,j}\frac{\Cij}{\Delta t} \prt*{ \big(p\ij\np\big)^2 - \big(p\ij\n)^2} + \frac{C}{\Delta t} \leq \frac{C}{\Delta t}.
\end{align}
Substituting \eqref{eq:R-final} into \eqref{eq:conv-estimate-interm}, we finally obtain
\begin{align*}
	\tilde e_h \leq C \prt*{Cv_h^{1/2} + v_h} h^{1/2},
\end{align*}
similar to the strategy of the weak-BV estimate in \cite{filbet}. Therefore, we have established 
\begin{align}
	\label{eq:eh_final}
	e_h \leq C \prt*{h^{1/2} + \Delta t + \norm{p_h(0) - p_0}_{L^1(Q)} + v_h h^{1/2}},
\end{align}
and thus
\[
	\int_{Q_T} p_h \bigg(\partialt \varphi + v \partialx \varphi - \prt{\Upsilon_p}_h \frac{\partial \varphi}{\partial v} \bigg) \dx t \dx x \dx v + \int_Q p_0\prt{x,v} \varphi\prt{0,x,v}\dx x \dx v \to 0, 
\]
as $h\to 0$, under condition \eqref{eq:clf} and for suitably chosen $v_h$ such that $v_h h^{1/2} \to 0$. Then the limit $\prt{f,g}$ of $\prt{f_h,g_h}$ is a weak solution to system \eqref{eq:main_system}, in the sense of Definition \ref{def:weak_sol}.
\end{proof}

\begin{rmk}[Equidistant meshes]
\label{rem:equidistant}
It is worthwhile pointing out that the rate of convergence in \eqref{eq:eh_final} can be improved if space and velocity are discretised equidistantly, i.e., there are $\Delta x, \Delta v>0$, such that $\Delta x_i=\Delta x$ and $\Delta v_j = \Delta v$, for all indices $i,j$. Indeed, this can be seen when we revisit the estimation of, e.g., the first equation in \eqref{eq:Jxh}. Assuming that $\Delta x_i = \Delta x$, we can estimate
\begin{align*}
	\mathcal J_x^{h,+} &= \sum_{n,j} \sum_{i=0}^{N_x-1}   \frac{[v\j]^+}{\Delta x} \prt{p\ij\n - p\imj\n} \varphi\ij\n\\
	&= - \sum_{n,j} \prt*{ \sum_{i=0}^{N_x-1} \frac{[v\j]^+}{\Delta x} \prt{\varphi\ipj\n - \varphi\ij\n} p\ij\n  + \frac{[v\j]^+}{\Delta x} \varphi_{0,j}^n p_{-1,j}^n},
\end{align*}
having used summation by parts and the fact that the test function is compactly supported. Note that the final term is  of order $\mathcal O(\Delta x)$. Hence
\begin{align*}
	\mathcal J_x^{h,+} &= - \prt*{\sum_{n,j} \sum_{i=0}^{N_x-1}   \frac{[v\j]^+}{\Delta x} p_{ij}^n \int_{C_{i,j}^n}(\varphi(t, x+ \Delta x, v) - \varphi(t, x, v))\dxvt} + \mathcal O(\Delta x)\\
	&= - \prt*{\sum_{n,j} \sum_{i=0}^{N_x-1}   [v\j]^+ p_{ij}^n \int_{C_{i,j}^n}\partial_x \varphi(t,x,v)  + \mathcal O(\Delta x)\dxvt}  + \mathcal O(\Delta x)\\
	&= - \prt*{\sum_{n,j} \sum_{i=0}^{N_x-1}    \int_{C_{i,j}^n}p_h(t,x,v) \, [v\j]^+  \partial_x \varphi(t,x,v) \dxvt}  + \mathcal O(\Delta x).
\end{align*}
This expression is almost identical to \eqref{eq:Jxplus-final}, except that now the error term, $\mathcal E_x^{h,+}$ (contributing an order $h^{1/2}$ in  \eqref{eq:eh_final}) is replaced by $\mathcal O (\Delta x)$ (i.e., of order $h$). The same type of estimate with similar contribution in the error term holds for $\mathcal I_x^{h,-},\mathcal I_v^{h,+}, \mathcal I_v^{h,-}$.
\end{rmk}

\begin{rmk}[Error estimates]
\label{rem:error_estimate}
Let $f_0, g_0 \in C^2(Q)$ be non-negative such that $\mathrm{supp}(p_0(x, \cdot)) \subset (-v_h, v_h)$, for $p_0\in \{f_0, g_0\}$. Moreover, assume that the CFL condition \eqref{eq:clf} is satisfied and that $K_{ij} \in W^{2,\infty} \prt{\Omega}$, for $i,j \in \{1,2\}.$
By \cite{CFI2022}, we know that the weak solutions in the sense of Definition \ref{def:weak_sol} to system \eqref{eq:main_system} are $\prt{f,g} \in C([0,T]; C^1(Q)^2)$ and remain compactly supported. From the equation, it follows that $f$ and $g$ are also Lipschitz in time and the strategy of the proof of \cite[Thm 5.1]{filbet} can be applied. Indeed, choosing $\varphi p\in W^{1,\infty}(Q_T)$ as a test function in the equation, we can establish the following convergence result
	\begin{equation}
	\begin{aligned}
	\label{eq:Error_estimate}
		&\norm{f - f_h }_{L^2(Q_T)}^2 +\norm{g - g_h }_{L^2(Q_T)}^2 \\ \quad & \leq C\prt*{\Delta t + h^{1/2} + \norm{f_0-f_h(0)}_{L^2(Q)}+\norm{g_0-g_h(0)}_{L^2(Q)}},
		\end{aligned}
	\end{equation}
    for some $C>0$ depending on $T$, $L$, $v_h$, $\C$, $\alpha$, $\lambda_1$, $\lambda_2$, $\norm{f}_{L^\infty \prt{Q_T}}$ and $\norm{g}_{L^\infty \prt{Q_T}}$. As a consequence of the compact support of the solution, note that we do not require $v_h \to \infty$, as $h\to 0$. Rather, it can be assumed to be fixed in a way that $\mathrm{supp}(p(t,x,\cdot)) \subset (-v_h,v_h)$.
    
    In light of Remark \ref{rem:equidistant}, we point out that the convergence order $h^{1/2}$ can be increased to $h$, should the space-time discretisation be equidistant.
\end{rmk}

\section{Numerical experiments}
\label{sec:experiment}
	Our next goal is to show the experimental order of convergence (EOC) in numerical examples. Since, to the best of our knowledge, there is no analytical solution to system \eqref{eq:main_system}, we  lack a reference solution for the computation of the error. Hence, we choose an approximated solution on a much finer grid as a reference to compute the errors. Note that, from now on, we compute approximations of the error, not the error itself.
	
	Let us consider the following interaction potentials
	\begin{align*}
	\pii(x) &\coloneqq \frac{x^2}{2}, &\pij(x) &\coloneqq \frac{x^2}{8},\\
	\pjj(x) &\coloneqq \frac{x^2}{2}, &\pji(x) &\coloneqq \frac{x^2}{8},
	\end{align*}
	corresponding to strong self-attraction and mild cross-attraction.
	We choose $L = 1$ and go to $T=1$ when we estimate the convergence in time and $T=1.375$ for the convergence in space. For this small time span, we can choose $v_h=5$. This allows us to create a fine mesh in both space and velocity to analyse the convergence in time and in phase space. As initial condition we set
	\begin{align*}
	f_0 &\coloneqq \begin{cases}
	\frac{99}{101}(0.5 + 0.5 \sin(\pi x)), &\text{ if } -1 \leq v \leq 1,\\
	\frac{99}{101 }(0.5 + 0.5 \sin(\pi x))\frac{1}{\abs{v}^{100}}, & \text{ else},
	\end{cases}\\
	g_0 &\coloneqq \begin{cases}
	\frac{99}{101}(0.5 - 0.5 \sin(\pi x)), &\text{ if } -1 \leq v \leq 1,\\
	\frac{99}{101}(0.5 - 0.5 \sin(\pi x))\frac{1}{\abs{v}^{100}}, & \text{ else}.
	\end{cases}
	\end{align*}

	In order to study the EOC we use different mesh refinements in time, space, and velocity. The mesh $\mathcal{T}_\ell^t$ at level $\ell\in \N$ is an equidistant discretisation of the interval $[0,T]$ with $\Delta t = 10^{-3}/2^{\ell}$. In other words, the time step $\Delta t$ is halved going from one level to the next. Concerning the phase-space discretisation, we define the meshes $\mathcal{T}_{\ell}^{xv}$ on each level $\ell\in \N$ by discretising space uniformly with $\Delta x = 2^{1-\ell} / 3$. The velocity domain is first split into an inner segment, $(-v_h/4, v_h/4)$, discretised uniformly with $\Delta v = 2^{-\ell-1}$, and an outer segment, $[-v_h, -v_h/4)\cup(v_h/4, v_h]$, discretised with $\Delta v = 15/2^{\ell+1}$.

	\subsection{Convergence in time}\label{sec:Con_time}
	First, we have a look at the convergence in time. For the mesh discretisation in time we consider a sequence of meshes $\{\mathcal{T}_{\ell}^t\}_{\ell=1}^{6}$, as introduced above.	 As reference solution we use $\prt{f_{\Delta t}^{\text{ref}},g_{\Delta t}^{\text{ref}}}$ computed on $\mathcal{T}_{8}^t$ with $\ell=8$. Additionally, we discretise space and velocity using the mesh $\mathcal{T}_{7}^{xv}$. This ensures that the initial condition is well discretised by the finer mesh in the middle. 
	
	Next, we would like to analyse the convergence of the solution for the mesh sequence $\{\mathcal{T}_{\ell}^t\}_{\ell=1}^{6}.$ For the numerical error we choose
	\begin{align*}
	\text{err}^1_\ell &\coloneqq \norm{f_{\Delta t}^{\text{ref}}-f_{\Delta t}}_{L^1(Q_T)}+\norm{g_{\Delta t}^{\text{ref}}-g_{\Delta t}}_{L^1(Q_T)},\\
	\text{err}^2_\ell &\coloneqq \norm{f_{\Delta t}^{\text{ref}}-f_{\Delta t}}^2_{L^2(Q_T)}+\norm{g_{\Delta t}^{\text{ref}}-g_{\Delta t}}^2_{L^2(Q_T)},
	\end{align*}
	where  $\prt{f_{\Delta t},g_{\Delta t}}$ is the piecewise constant solution of the discrete scheme \eqref{eq:finite-volume-method} at level $\ell$ and $\prt{f_{\Delta t}^{\text{ref}},g_{\Delta t}^{\text{ref}}}$ is the solution at the reference level $\ell=8$. \\	
	\begin{minipage}{.475\textwidth}
		To estimate the order of convergence we use the EOC, defined as
		\[
		\text{EOC} = \frac{\ln(\text{err}_{\ell-1}) -\ln(\text{err}_{\ell})}{\ln(\Delta t_{\ell-1 }) -\ln(\Delta t_{\ell })}.
		\]
		The associated $L^1(Q_T)$-errors (resp. squared $L^2(Q_T)$-errors) are presented in Table \ref{tab:Delta_conv_E1} (resp.  in Table \ref{tab:Delta_conv_E2}). We see that the $L^1(Q_T)$-errors converge with an experimental order $1$ with respect to $\Delta t$, cf. Table \ref{tab:Delta_conv_E1}, whereas the squared $L^2(Q_T)$-errors converge with an experimental order $2$, see Table \ref{tab:Delta_conv_E2}. This is visualised in the log-log plot Figure \ref{pic:convergence_E1}, which depicts the error vs. time step size. A second order gradient triangle is superimposed as reference. Note that, compared to the next section, the error in the time discretisation appears to be dominant.
	\end{minipage}
	\begin{minipage}{.525\textwidth}
		\centering
		\begin{tikzpicture}
		\begin{axis}[
		height=.65\textwidth,
		width=.65\textwidth,
		axis x line=bottom,
		axis y line=left,
		tick align=outside,
		tickpos=left,
		xlabel=$\Delta t$,
		xmax =5e-4,
		ymax = 62850,
		xmode=log,
		ymode=log,
		]
		\addplot table {
			0.0005 62808.2825733932
			0.00025 14687.9714253431
			0.000125 3473.51526208418
			0.0000625 804.564471466556
			0.00003125 174.563011261167
			0.000015625 32.1319637587428
		};
		\addplot[mark=none] table {
		0.000125	14800
		0.00003125		925
		0.00003125		14800
		0.000125		14800		
	};
		\end{axis}
		\end{tikzpicture}
		
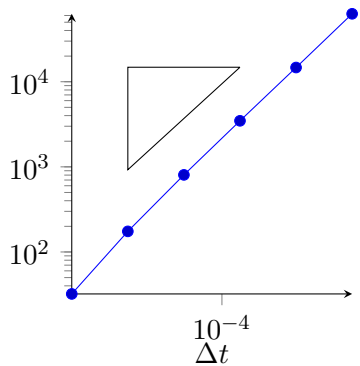
\captionof{figure}{Squared $L^2(Q_T)$-errors (blue) convergence in $\Delta t$ with order 2 (triangle) for fixed $h$.}\label{pic:convergence_E1}
	\end{minipage}
	\begin{center}
		\begin{table}[h]
			\begin{tabular}{c|c|r|c|c|c|c|c}
				$\ell$	& \#$xv$-cells&	 \#$txv$-cells&		$\Delta t$	&$\alpha h$& $h$&$\text{err}^1_\ell$ &	EOC\\
				\hline
			1&	294912&	5.90$\times 10^{8}$& 5$\times 10^{-4}$&	5.21$\times 10^{-3}$&	5.86$\times 10^{-2}$&	7.54033$\times 10^{1}$&	-\\
			2&	294912&	1.18$\times 10^{9}$&	2.5$\times 10^{-4}$&		5.21$\times 10^{-3}$&	5.86$\times 10^{-2}$ &	3.70549$\times 10^{1}$&	1.02\\
			3&	294912&	2.36$\times 10^{9}$&	1.25$\times 10^{-4}$&		5.21$\times 10^{-3}$&	5.86$\times 10^{-2}$ &	1.81502$\times 10^{1}$&	1.03\\
			4&	294912&	4.72$\times 10^{9}$&	6.25$\times 10^{-5}$&		5.21$\times 10^{-3}$&	5.86$\times 10^{-2}$ &	8.76259$\times 10^{0}$&	1.05\\
			5&	294912&	9.45$\times 10^{9}$&	3.13$\times 10^{-5}$&		5.21$\times 10^{-3}$&	5.86$\times 10^{-2}$&	4.08463$\times 10^{0}$&	1.10\\
			6&	294912&	1.89$\times 10^{10}$&	1.56$\times 10^{-5}$&		5.21$\times 10^{-3}$&	5.86$\times 10^{-2}$&	1.74958$\times 10^{0}$&	1.22
			\end{tabular}
			\caption{$L^1(Q_T)$-error convergence table of the discrete scheme \eqref{eq:finite-volume-method} in $\Delta t$.}\label{tab:Delta_conv_E1}
		\end{table}
	\end{center}
	\begin{center}
		\begin{table}[h]
			\begin{tabular}{c|c|r|c|c|c|c|c}
				$\ell$	& \#$xv$-cells&	 \#$txv$-cells&		$\Delta t$	&$\alpha h$& $h$&$\text{err}^2_\ell$ &	EOC\\
				\hline
					1&	294912&	5.90$\times 10^{8}$& 5$\times 10^{-4}$&5.21$\times 10^{-3}$&	5.86$\times 10^{-2}$&	6.28083$\times 10^{4}$ 	&	-\\
				2&	294912&	1.18$\times 10^{9}$&2.5$\times 10^{-4}$&5.21$\times 10^{-3}$&	5.86$\times 10^{-2}$ &1.46880$\times 10^{4}$	&	2.10				\\
				3&	294912&	2.36$\times 10^{9}$&1.25$\times 10^{-4}$&5.21$\times 10^{-3}$&	5.86$\times 10^{-2}$ & 3.47352$\times 10^{3}$	&2.08				\\
				4&	294912&	4.72$\times 10^{9}$&6.25$\times 10^{-5}$&5.21$\times 10^{-3}$&	5.86$\times 10^{-2}$ &	8.04564$\times 10^{2}$	&2.11				\\
				5&	294912&	9.45$\times 10^{9}$&3.13$\times 10^{-5}$&5.21$\times 10^{-3}$&	5.86$\times 10^{-2}$&1.74563$\times 10^{2}$	&	2.20				\\
				6&	294912&	1.89$\times 10^{10}$&1.56$\times 10^{-5}$&5.21$\times 10^{-3}$&	5.86$\times 10^{-2}$&3.21320$\times 10^{1}$	&	2.44			
			\end{tabular}
					\caption{Squared $L^2(Q_T)$-error convergence table of the discrete scheme \eqref{eq:finite-volume-method} in $\Delta t$.}\label{tab:Delta_conv_E2}
		\end{table}
	\end{center}
	
	\subsection{Convergence in space and velocity}
	Let us now have a look at the convergence in space and velocity. Hence, we fix a time step, $\Delta t = 5\times 10^{-5}$, and consider the sequence of meshes $\{\mathcal{T}_{\ell}^{xv}\}_{\ell=1}^{7}$ in space and velocity, as introduced above. Compared to the previous subsection, we now use a new reference solution on the fine phase-space grid, $\mathcal{T}_{9}^{xv}$, denoted by $(f_h^{\text{ref}},g_h^{\text{ref}})$.
For the numerical errors we define
	\begin{align*}
	\text{err}^1_\ell \coloneqq \norm{f_h^{\text{ref}}-f_h}_{L^1(Q_T)}+\norm{g_h^{\text{ref}}-g_h}_{L^1(Q_T)},\\
	\text{err}^2_\ell \coloneqq \norm{f_h^{\text{ref}}-f_h}^2_{L^2(Q_T)}+\norm{g_h^{\text{ref}}-g_h}^2_{L^2(Q_T)},
	\end{align*}
	where  $\prt{f_h,g_h}$ is the piecewise constant solution of the discrete scheme \eqref{eq:finite-volume-method} for the discretisation at level $\ell$ and $\prt{f_h^{\text{ref}},g_h^{\text{ref}}}$ the solution at level $\ell=9$. Following this scheme we compute the $L^1(Q_T)$-errors and squared $L^2(Q_T)$-errors as presented in Table \ref{tab:h_conv_E1} and Table \ref{tab:h_conv_E2}. 
		\begin{center}
		\begin{table}[h]
			\begin{tabular}{c|c|c|c|c|c|c|c}
				$\ell$	& \#$xv$-cells&	 \#$txv$-cells&		$\Delta t$	&$\alpha h$& $h$&$\text{err}^1_\ell$ &	EOC\\
				\hline
			1	& 72	&1.98$\times 10^{6} $&5$\times 10^{-5}$	&3.33$\times 10^{-1}$	&3.75$\times 10^{0}$		&3.40177	&-\\
			2	&288	&7.92$\times 10^6$	&5$\times 10^{-5}$	&1.67$\times 10^{-1}$	&1.88$\times 10^{0}$		&2.57744	&0.40\\
			3	&1152	&3.17$\times 10^7$	&5$\times 10^{-5}$	&8.33$\times 10^{-2}$	&9.38$\times 10^{-1}$		&1.84553	&0.48\\
			4	&4608	&1.27$\times 10^8$	&5$\times 10^{-5}$	&4.17$\times 10^{-2}$	&4.69$\times 10^{-1}$	&1.28177	&0.53\\
			5	&18432	&5.07$\times 10^8$	&5$\times 10^{-5}$	&2.08$\times 10^{-2}$ 	&2.34$\times 10^{-1}$	&0.84764 	&0.60 \\
			6	&73728	&2.03$\times 10^9$	&5$\times 10^{-5}$	&1.04$\times 10^{-2}$	&1.17$\times 10^{-1}$	&0.52654	&0.69\\
			7	&294912	&8.11$\times 10^9$	&5$\times 10^{-5}$	&5.21$\times 10^{-3}$	&5.86$\times 10^{-2}$ & 0.29243	& 0.85
			\end{tabular}
			\caption{$L^1$-error convergence table of the discrete scheme \eqref{eq:finite-volume-method} in $h$ on a non-equidistant mesh.}\label{tab:h_conv_E1}
		\end{table}
	\end{center}
		\begin{center}
		\begin{table}[h]
			\begin{tabular}{c|c|c|c|c|c|c|c}
				$\ell$	& \#$xv$-cells&	 \#$txv$-cells&		$\Delta t$	&$\alpha h$& $h$&$\text{err}^2_\ell$ &	EOC\\
				\hline
				1	& 72	&1.98$\times 10^{6} $&5$\times 10^{-5}$	&3.33$\times 10^{-1}$	&3.75$\times 10^{0}$		&0.77771	&-\\
				2	&288	&7.92$\times 10^6$	&5$\times 10^{-5}$	&1.67$\times 10^{-1}$	&1.88$\times 10^{0}$		&0.49549	&0.65\\
				3	&1152	&3.17$\times 10^7$	&5$\times 10^{-5}$	&8.33$\times 10^{-2}$	&9.38$\times 10^{-1}$		&0.31754	&0.64\\
				4	&4608	&1.27$\times 10^8$	&5$\times 10^{-5}$	&4.17$\times 10^{-2}$	&4.69$\times 10^{-1}$	&0.19797	&0.68\\
				5	&18432	&5.07$\times 10^8$	&5$\times 10^{-5}$	&2.08$\times 10^{-2}$ 	&2.34$\times 10^{-1}$	&0.11015	&0.85 \\
				6	&73728	&2.03$\times 10^9$	&5$\times 10^{-5}$	&1.04$\times 10^{-2}$	&1.17$\times 10^{-1}$	&0.05375&1.04\\
				7	&294912	&8.11$\times 10^9$	&5$\times 10^{-5}$	&5.21$\times 10^{-3}$	&5.86$\times 10^{-2}$	&0.02077	&1.37
			\end{tabular}
			\caption{Squared $L^2(Q_T)$-error convergence table of the discrete scheme \eqref{eq:finite-volume-method} in $h$ on a non-equidistant mesh.}\label{tab:h_conv_E2}
		\end{table}
	\end{center}
	\begin{minipage}{.5\textwidth}
	Let us have a closer look at the error table Table \ref{tab:h_conv_E2}. First, we see that there is a slightly worse error convergence from the second to the fourth level. This is likely due to the fact that the maximal space and velocity step $h$ is greater than or close to one. Once $h$ is sufficiently small we see that the estimated order of convergence picks up and even exceeds the analytical order of convergence in space and velocity which is $0.5$, cf. Remark \ref{rem:error_estimate}. This can be seen as well in the log-log plot Figure \ref{pic:convergence_E_h} which plots the errors of Table \ref{tab:h_conv_E2} with respect to $h$ and shows a reference triangle with slop one. If we consider the convergence in the $L^1(Q_T)$-norm in Table \ref{tab:h_conv_E1}, we see that the order of convergence is definitely below one.
\end{minipage}
		\begin{minipage}{.5\textwidth}
			\centering
		\hspace*{.2cm}\begin{tikzpicture}
		\begin{axis}[
		height=.7\textwidth,
		width=.7\textwidth,
		axis x line=bottom,
		axis y line=left,
		tick align=outside,
		tickpos=left,
		xlabel=$h$,
		xmax =1,
		ymax =1.2,
		xmin = 0.05,
		ymin = 0.02,
		xmode=log,
		ymode=log,
		]
		\addplot table {  
			0.9375		0.31754
			0.46875		0.19797
			0.234375	0.11015 	
			0.1171875	0.05375
			0.05859375	0.02077				
		};
		\addplot[mark=none] table {
			0.6		0.6
			0.1		0.1
			0.1		0.6
			0.6		0.6			
		};
		\end{axis}
		\end{tikzpicture}
		
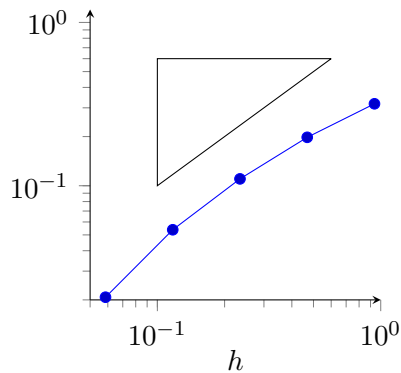
\captionof{figure}{Squared $L^2(Q_T)$-error convergence in $h$ for fixed $\Delta t$.}\label{pic:convergence_E_h}
	\end{minipage}
\subsection{Convergence in space and velocity - an equidistant mesh}
As final example, we consider an equidistant mesh in each space and velocity. In this case, by Remark \ref{rem:error_estimate}, we expect convergence of order one for $\text{err}^2_\ell$ as $h \to 0$. \\
\begin{minipage}{.5\textwidth}
		For the mesh we fix again a time step, $\Delta t = 5\times 10^{-5}$, and consider a sequence of meshes $\{\mathcal{T}_{\ell}^{xv}\}_{\ell=1}^{7}$ in space and velocity. For the spatial mesh, we will consider the same mesh as before, namely on level $\ell=1$ we start with an equidistant mesh with $\Delta x = 1/3 $. In velocity we now consider an equidistant mesh with  $\Delta v = 5/6$ for $\ell=1$ and refine by using bisection in both space and velocity. As reference solution we consider the solution computed on $\mathcal T_9^{xv}$. In our results, we see in both the $L^1(Q_T)$-norm and the squared $L^2(Q_T)$-norm an improvement in the order of convergence, cf. Table \ref{tab:h_equid_conv_E1} and Table \ref{tab:h_equid_conv_E2} (in comparison with Table \ref{tab:h_conv_E1} and Table \ref{tab:h_conv_E2}, respectively). However, in  Table \ref{tab:h_equid_conv_E2}  we see a much quicker convergence to the analytical convergence order, that is $1$, compared to the error convergence on a non-equidistant mesh, see Table \ref{tab:h_conv_E2}.
\end{minipage}
\begin{minipage}{.5\textwidth}
\centering
\hspace*{.2cm}\begin{tikzpicture}
	\begin{axis}[
		height=.7\textwidth,
		width=.7\textwidth,
		axis x line=bottom,
		axis y line=left,
		tick align=outside,
		tickpos=left,
		xlabel=$h$,
		xmax =1,
		ymax =1.2,
		xmin = 0.01,
		ymin = 0.02,
		xmode=log,
		ymode=log,
		]
		\addplot table { 
			0.83333333334		1.18741
			0.41666666667		0.78554
			0.20833333334		0.36358
			0.10416666667		0.20935
			0.05208333336	0.11281 	
			0.02604166668	0.05332
			0.01302083334	0.02052				
		};
		\addplot[mark=none] table {
			0.28		0.28
			0.04		0.04
			0.28		0.04
			0.28		0.28			
		};
	\end{axis}
\end{tikzpicture}

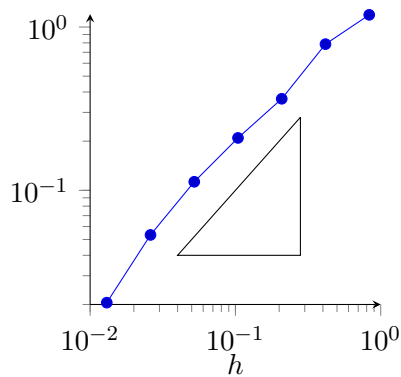
\captionof{figure}{Squared $L^2(Q_T)$-error convergence in $h$ for fixed $\Delta t$ on an equidistant mesh}\label{pic:convergence_E_h_equid}
\end{minipage}\\
		\begin{center}
	\begin{table}[h]
		\begin{tabular}{c|c|c|c|c|c|c|c}
			$\ell$	& \#$xv$-cells&	 \#$txv$-cells&		$\Delta t$	&$\alpha h$& $h$&$\text{err}^1_L$ &	EOC\\
			\hline
			1	& 72	&1.98$\times 10^{6} $&5$\times 10^{-5}$&3.33$\times 10^{-1}$  	&8.33$\times 10^{-1}$		&3.65920&-\\
			2	&288	&7.92$\times 10^6$	&5$\times 10^{-5}$	&1.67$\times 10^{-1}$	&4.17$\times 10^{-1}$	&2.77739&0.40\\
			3	&1152	&3.17$\times 10^7$	&5$\times 10^{-5}$	&8.33$\times 10^{-2}$	&2.08$\times 10^{-1}$	&2.03670&0.45\\
			4	&4608	&1.27$\times 10^8$	&5$\times 10^{-5}$	&4.17$\times 10^{-2}$	&1.04$\times 10^{-1}$	&1.37418&0.56\\
			5	&18432	&5.07$\times 10^8$	&5$\times 10^{-5}$  &2.08$\times 10^{-2}$ 	&5.21$\times 10^{-2}$	&0.86499	&0.67 \\
			6	&73728	&2.03$\times 10^9$	&5$\times 10^{-5}$	&1.04$\times 10^{-2}$	&2.60$\times 10^{-2}$	&0.52418	&0.72\\
			7	&294912	&8.11$\times 10^9$	&5$\times 10^{-5}$	&5.21$\times 10^{-3}$	&1.30$\times 10^{-2}$	&0.28511	&0.88
		\end{tabular}
		\caption{$L^1(Q_T)$-error convergence table of the discrete scheme \eqref{eq:finite-volume-method} in $h$ on an equidistant mesh}\label{tab:h_equid_conv_E1}
	\end{table}
\end{center}
\begin{center}
	\begin{table}[h]
		\begin{tabular}{c|c|c|c|c|c|c|c}
			$\ell$	& \#$xv$-cells&	 \#$txv$-cells&	$\Delta t$	&$\alpha h$& $h$&$\text{err}^2_L$ &	EOC\\
			\hline
			1	& 72	&1.98$\times 10^{6} $&5$\times 10^{-5}$&3.33$\times 10^{-1}$  	&8.33$\times 10^{-1}$			&1.18741	&-\\
			2	&288	&7.92$\times 10^6$	&5$\times 10^{-5}$	&1.67$\times 10^{-1}$	&4.17$\times 10^{-1}$	&0.78554	&0.60\\
			3	&1152	&3.17$\times 10^7$	&5$\times 10^{-5}$	&8.33$\times 10^{-2}$	&2.08$\times 10^{-1}$	&0.36358	&1.11\\
			4	&4608	&1.27$\times 10^8$	&5$\times 10^{-5}$	&4.17$\times 10^{-2}$	&1.04$\times 10^{-1}$	&0.20935	&0.80\\
			5	&18432	&5.07$\times 10^8$	&5$\times 10^{-5}$  &2.08$\times 10^{-2}$ 	&5.21$\times 10^{-2}$	&0.11281	&0.89 \\
			6	&73728	&2.03$\times 10^9$	&5$\times 10^{-5}$	&1.04$\times 10^{-2}$	&2.60$\times 10^{-2}$	&0.05332 	&1.08\\
			7	&294912	&8.11$\times 10^9$	&5$\times 10^{-5}$	&5.21$\times 10^{-3}$	&1.30$\times 10^{-2}$	&0.02052	&1.37
		\end{tabular}
		\caption{Squared $L^2(Q_T)$-error convergence table of the discrete scheme \eqref{eq:finite-volume-method} in $h$ on an equidistant mesh}\label{tab:h_equid_conv_E2}
	\end{table}
\end{center}
We observe a more stable convergence in the case of an equidistant mesh and a linear convergence $h$ in the squared $L^2(Q_T)$-norm. From our experiments we infer that the optimal order of convergence in $h$ is one. Additionally, the findings in Section~\ref{sec:Con_time} suggest an optimal order of convergence of two in $\Delta t$. Further study of this behaviour will be necessary to derive the optimal convergence result in  $\Delta t$.
	
	\section{Conclusion and outlook}
	In this paper, we have derived a finite volume scheme which preserves the mass conservation, positivity and boundedness of convex functionals. Additionally, we showed in Theorem \ref{th:convergence} that the $L^\infty(Q_T)$ weak-$*$ limit $(f,g)$ of our approximation $(f_h,g_h)$ is also the weak solution of the original system (\ref{eq:main_system}), as $h\to 0$. Subsequently, we established a convergence estimate in the squared $L^2(Q_T)$-norm and had a look at the corresponding convergence in our numerical investigation. In the case of equidistant discretisation in space and velocity,  we observed a convergence order of one as predicted in Remark \ref{rem:error_estimate}. 
	In addition, our experiments displayed a  convergence order of two in time suggesting the convergence order results in Remark \ref{rem:error_estimate} may be improved with respect to $\Delta t$.
	
	Our scheme acts as a solid foundation for further studies including models exhibiting damping. It is well-known that an effective macroscopic system is obtained in the overdamped limit which matches the system studied in \cite{difrancesco_fagioli}. Therefore, it would be of great interest to study whether or not the scheme is asymptotic preserving, i.e., stable with respect to the damping parameter. Another challenging, open problem is to investigate analytically if the time convergence order can be improved.

\section*{Acknowledgements}
JIMH acknowledges funding within the Postdoc Starter Kit by the Graduate Academy of Dresden University of Technology financed by the Federal Ministry of Education and Research (BMBF) and the State of Saxony under the Excellence Strategy of the German federal and state governments. VI would like to acknowledge the invitation to Dresden where parts of this project were established. The research of VI is supported by the Italian INdAM project N. E53C22001930001 ``MMEAN-FIELDS.''

\bibliographystyle{plain}
\bibliography{literature}
\end{document}